\documentclass{article}

\usepackage{iclr2021_conference,times}
\usepackage[utf8]{inputenc} %
\usepackage[T1]{fontenc}    %
\usepackage{hyperref}       %
\usepackage{url}            %
\usepackage{booktabs}       %
\usepackage{amsfonts}       %
\usepackage{nicefrac}       %
\usepackage{microtype}      %
\usepackage[user,titleref]{zref}
\usepackage{amsmath,amsfonts,amsthm} %
\usepackage{amssymb}
\usepackage{enumerate}
\usepackage{xcolor}

\newtheorem{theorem}{Theorem}
\newtheorem{prop}{Proposition}
\newtheorem{lemma}{Lemma}
\newtheorem{definition}{Definition}
\def\E{{\mathbb E}}

\def\R{{\mathbb R}}
\def\N{{\mathbb N}}

\def\E{{\mathbb E}}
\def\P{{\mathbb P}}

\def\ii{{\rm i}}

\def\dd{\mathrm{d}}

\newenvironment{sistema}%
  {\left\lbrace\begin{array}{@{}l@{}}}%
  {\end{array}\right.}

\hypersetup{%
  pdfpagemode={UseOutlines},
  bookmarksopen,
  pdfstartview={FitH},
  colorlinks,
  linkcolor={blue},
  citecolor={blue},
  urlcolor={green}
}

\title{Large-width functional asymptotics for deep Gaussian neural networks}

\author{Daniele Bracale$^1$, Stefano Favaro$^{1,2}$, Sandra Fortini$^3$, Stefano Peluchetti$^4$ \\
$^1$ University of Torino, $^2$ Collegio Carlo Alberto, $^3$ Bocconi University, $^4$ Cogent Labs
}

\iclrfinalcopy
\begin{document}
\maketitle
\begin{abstract}In this paper, we consider fully-connected feed-forward deep neural networks where weights and biases are independent and identically distributed according to Gaussian distributions. Extending previous results \citep{matthews2018gaussian,matthews2018gaussianB,yang2019wide} we adopt a function-space perspective, i.e. we look at 
neural networks as infinite-dimensional random elements on the input space $\R^I$. Under suitable assumptions on the activation function we show that: i) a network defines a continuous stochastic process on the input space $\mathbb{R}^I$; ii) a network with re-scaled weights converges weakly to a continuous Gaussian Process in the large-width limit; iii) the limiting Gaussian Process has almost surely locally $\gamma$-H\"older continuous paths, for $0 < \gamma <1$. Our results contribute to recent theoretical studies on the interplay between infinitely-wide deep neural networks and Gaussian Processes by establishing weak convergence in function-space with respect to a stronger metric.
\end{abstract}

\section{Introduction}\label{sec:intro}

The interplay between infinitely-wide deep neural networks and classes of Gaussian Processes has its origins in the seminal work of \cite{neal1995bayesian}, and it has been the subject of several theoretical studies. See, e.g., \cite{der2006beyond}, \citet{lee2018deep}, \cite{matthews2018gaussian,matthews2018gaussianB}, \citet{yang2019wide} and references therein. Let consider a fully-connected feed-forward neural network with re-scaled weights composed of $L\geq1$ layers of widths $n_1,\dots,n_L$, i.e.
\begin{equation}\label{neural_start}
\begin{aligned}
&f_{i}^{(1)}(x)=\sum_{j=1}^{I}w_{i,j}^{(1)}x_{j}+b_{i}^{(1)}\quad&& i=1,\dots,n_1\\
&f_{i}^{(l)}(x)=\frac{1}{\sqrt{n_{l-1}}}\sum_{j=1}^{n_{l-1}}w_{i,j}^{(l)}\phi(f_{j}^{(l-1)}(x))+b_{i}^{(l)}\quad&&l=2,\ldots,L,\ \ i=1,\dots,n_l
\end{aligned}
\end{equation}

where $\phi$ is a non-linearity and $x \in \mathbb{R}^I$ is a real-valued input of dimension $I \in \N$. \cite{neal1995bayesian} considered the case $L=2$, a finite number $k\in \N$ of fixed distinct inputs $(x^{(1)},\ldots,x^{(k)})$, with each $x^{(r)} \in \mathbb{R}^I$, and weights $w_{i,j}^{(l)}$ and biases $b_{i}^{(l)}$ independently and identically distributed (iid) as Gaussian distributions. Under appropriate assumptions on the activation $\phi$ \cite{neal1995bayesian} showed that: i) for a fixed unit $i$, the $k$-dimensional random vector $(f^{(2)}_i(x^{(1)}),\dots,f^{(2)}_i(x^{(k)}))$ converges in distribution, as the width $n_1$ goes to infinity, to a $k$-dimensional Gaussian random vector; ii) the large-width convergence holds jointly over finite collections of $i$'s and the limiting $k$-dimensional Gaussian random vectors are independent across the index $i$. These results concerns neural networks with a single hidden layer, but \cite{neal1995bayesian} also includes preliminary considerations on infinitely-wide deep neural networks. More recent works, such as \cite{lee2018deep}, established convergence results corresponding to \cite{neal1995bayesian} results i) and ii) for deep neural networks under the assumption that widths $n_1,\dots,n_L$ go to infinity sequentially over network layers. \cite{matthews2018gaussian,matthews2018gaussianB} extended the work of \cite{neal1995bayesian,lee2018deep} by assuming that the width $n$ grows to infinity jointly over network layers, instead of sequentially, and by establishing joint convergence over all $i$ and countable distinct inputs. The joint growth over the layers is certainly more realistic than the sequential growth, since the infinite Gaussian limit is considered as an approximation of a very wide network. We operate in the same setting of \cite{matthews2018gaussianB}, hence from here onward $n \geq 1$ denotes the common layer width, i.e. $n_1,\dots,n_L = n$. Finally, similar large-width limits have been established for a great variety of neural network architectures, see for instance \cite{yang2019wide}.

The assumption of a countable number of fixed distinct inputs is the common trait of the literature on large-width asymptotics for deep neural networks. Under this assumption, the large-width limit of a network boils down to the study of the large-width asymptotic behavior of the $k$-dimensional random vector $(f^{(l)}_i(x^{(1)}),\dots,f^{(l)}_i(x^{(k)}))$ over $i\geq1$ for finite $k$. Such limiting finite-dimensional distributions describe the large-width distribution of a neural network a priori over any dataset, which is finite by definition. When the limiting distribution is Gaussian, as it often is, this immediately paves the way to Bayesian inference for the limiting network. Such an approach is competitive with the more standard stochastic gradient descent training for the fully-connected architectures object of our study \citep{lee2020finite}. However, knowledge of the limiting finite-dimensional distributions is not enough to infer properties of the limiting neural network which are inherently uncountable such as the continuity of the limiting neural network, or the distribution of its maximum over a bounded interval. Results in this direction give a more complete understanding of the assumptions being made a priori, and hence whether a given model is appropriate for a specific application. For instance, \cite{van2011information} shows that for Gaussian Processes the function smoothness under the prior should match the smoothness of the target function for satisfactory inference performance. 

In this paper we thus consider a novel, and more natural, perspective to the study of large-width limits of deep neural networks. This is an infinite-dimensional perspective where, instead of fixing a countable number of distinct inputs, we look at $f^{(l)}_i(x,n)$ as a stochastic process over the input space $\R^I$. Under this perspective, establishing large-width limits requires considerable care and, in addition, it requires to show the existence of both the stochastic process induced by the neural network and its large-width limit. We start by proving the existence of i) a continuous stochastic process, indexed by the network width $n$, corresponding to the fully-connected feed-forward deep neural network; ii) a continuous Gaussian Process corresponding to the infinitely-wide limit of the deep neural network. Then, we prove that the stochastic process i) converges weakly, as the width $n$ goes to infinity, to the Gaussian Process ii) jointly over all units $i$. As a by-product of our results, we show that the limiting Gaussian Process has almost surely locally $\gamma$-H\"older continuous paths, for $0 < \gamma <1$. To make the exposition self-contained we include an alternative proof of the main result of \cite{matthews2018gaussian,matthews2018gaussianB}, i.e. the finite-dimensional limit for full-connected neural networks. The major difference between our proof and that of \cite{matthews2018gaussianB} is due to the use of the characteristic function to establish convergence in distribution, instead of relying on a CLT \citep{blum1958central} for exchangeable sequences.

The paper is structured as follows. In Section \ref{sec:setting} we introduce the setting under which we operate, whereas in  Section \ref{sec:sketch} we present a high-level overview of the approach taken to establish our results. Section \ref{sec:limit} contains the core arguments of the proof of our large-width functional limit for deep Gaussian neural networks, which are spelled out in detail in the supplementary material (SM). We conclude in Section \ref{sec:discussion}.

\section{Setting}\label{sec:setting}

Let $(\Omega,\mathcal{H},\mathbb{P})$ be the probability space on which all random elements of interest are defined. Furthermore, let $N(\mu,\sigma^2)$ denote a Gaussian distribution with mean $\mu \in \R$ and strictly positive variance $\sigma^2 \in \R_{+}$, and let $N_k(\mathbf{m},\Sigma)$ be a $k$-dimensional  Gaussian distribution with mean $\mathbf{m} \in \R^k$ and covariance matrix $\Sigma \in \R^{k \times k}$. In particular, $\R^k$ is equipped with $\| \cdot \|_{\R^k}$, the euclidean norm induced by the inner product $\langle \cdot, \cdot \rangle_{\R^k}$, and $\R^{\infty}=\times_{i=1}^{\infty}\R$ is equipped with $\| \cdot \|_{\R^{\infty}}$, the norm induced by the distance $d(\textbf{a},\textbf{b}) _{\infty}=\sum_{i\geq1}\xi(|a_i-b_i|)/2^{i}$ for $\textbf{a},\textbf{b} \in \R^{\infty}$ (Theorem 3.38 of \cite{charalambos2013infinite}), where $\xi(t)=t/(1+t)$ for all real values $t \geq 0$. Note that $(\R,|\cdot|)$ and $(\R^{\infty},\| \cdot \|_{\R^{\infty}})$ are Polish spaces, i.e. separable and complete metric spaces (Corollary 3.39 of \cite{{charalambos2013infinite}}). We choose $d_{\infty}$ since it generates a topology that coincides with the product topology (line 5 of the proof of Theorem 3.36 of \cite{charalambos2013infinite}). The space $(S,d)$ will indicate a generic Polish space such as $\R$ or $\R^{\infty}$ with the associated distance. We indicate with $S^{\R^I}$ the space of functions from $\R^I$ into $S$ and $C(\R^I;S) \subset S^{\R^I}$ the space of continuous functions from $\R^I$ into $S$. Let $\omega_{i,j}^{(l)}$ be the random weights of the $l$-th layer, and assume that they are iid as $N(0,\sigma_{\omega}^2)$, i.e. 
\begin{align}
\label{caromega}
\varphi_{ \omega_{i,j}^{(l)} } (t) = \E [ e^{\ii t \omega_{i,j}^{(l)} } ] = e^{-\frac{1}{2}\sigma_{\omega}^2t^2}
\end{align}
is the characteristic function of $\omega_{i,j}^{(l)} $, for $i \geq 1$, $j=1, \dots, n$ and $l \geq 1$. Let $b^{(l)}_i$ be the random biases of the $l$-th layer, and assume that they are iid as $N(0,\sigma_b^2)$, i.e.
\begin{align}
\label{carbias}
\varphi_{ b_{i}^{(l)} } (t) = \E [ e^{\ii t b_{i}^{(l)} } ] = e^{-\frac{1}{2}\sigma_{b}^2t^2}
\end{align}
is the characteristic function of $b_{i}^{(l)}$, for $i \geq 1$ and $l \geq 1$. Weights $\omega^{(l)}_{i,j}$ are independent of biases $b_{i}^{(l)}$, for any $i \geq 1, j=1,\dots, n$ and $l \geq 1$. Let $\phi: \R \to\ \R$ denote a continuous non-linearity. For the finite-dimensional limit we will assume the polynomial envelop condition
\begin{align}
\label{envelope}
|\phi(s)| \leq a+b|s|^{m},
\end{align}
for any $s \in \R$ and some real values $a,b>0$ and $m \geq 1$. For the functional limit we will use a stronger assumption on $\phi$, assuming $\phi$ to be Lipschitz on $\R$ with Lipschitz constant $L_{\phi}$.

Let $Z$ be a stochastic process on $\mathbb{R}^{I}$, i.e. for each $x\in\mathbb{R}^{I}$, $Z(x)$ is defined on $(\Omega,\mathcal{H},\mathbb{P})$ and it takes values in $S$. For any $k\in \N$ and $x_{1},\ldots,x_{k}\in\mathbb{R}^{I}$, let $P^Z_{x_{1},\ldots,x_{k}}=\P(Z(x_{1})\in A_{1},\ldots, Z(x_{k})\in A_{k})$, with $A_{1},\ldots,A_{k}\in\mathcal{B}(S)$. Then, the family of finite-dimensional distributions of $Z(x)$ is defined as the family of distributions $\{P^Z_{x_{1},\ldots,x_{k}}\text{ : }x_{1},\ldots,x_{k}\in\mathbb{R}^{I}\text{ and }k\in \N\}$. See, e.g., \cite{billingsley1995probability}. In Definition \ref{def:1} and Definition \ref{def:2} we look at the deep neural network \eqref{neural_start} as a stochastic process on input space $\mathbb{R}^{I}$, that is a stochastic process whose finite-dimensional distributions are determined by a finite number $k\in \N$ of fixed distinct inputs $(x^{(1)},\ldots,x^{(k)})$, with each $x^{(r)} \in \mathbb{R}^I$. The existence of the stochastic processes of Definition \ref{def:1} and Definition \ref{def:2} will be thoroughly discussed in Section \ref{sec:sketch}.

\begin{definition}\label{def:1}
For any fixed $l\geq 2$ and $i \geq 1$, let $(f_i^{(l)}(n))_{n\geq 1}$ be a sequence of stochastic processes on $\R^{I}$. That is, $f_i^{(l)}(n): \R^I \rightarrow \R$, with $x \mapsto f_i^{(l)}(x,n)$, is a stochastic process on $\R^{I}$ whose finite-dimensional distributions are the laws, for any $k\in \N$ and $x^{(1)},\ldots,x^{(k)}\in\R^{I}$, of the $k$-dimensional random vectors
\begin{equation}\label{f1}
f^{(1)}_i({\textbf{X},n})=f^{(1)}_i({\textbf{X}})=[f^{(1)}_i(x^{(1)},n), \dots , f^{(1)}_i(x^{(k)},n)]^T=\sum_{j=1}^{I}\omega_{i,j}^{(1)}\textbf{x}_{j}+ b_{i}^{(1)}\textbf{1}
\end{equation}
\begin{equation}\label{f2}
f^{(l)}_i({\textbf{X},n})=[f^{(l)}_i(x^{(1)},n), \dots , f^{(l)}_i(x^{(k)},n)]^T=\frac{1}{\sqrt{n}}\sum_{j=1}^{n}\omega_{i,j}^{(l)}( \phi \bullet f^{(l-1)}_j({\textbf{X},n}))+ b_{i}^{(l)}\textbf{1}
\end{equation}
where $\textbf{X}=[ x^{(1)}, \dots , x^{(k)}] \in \R^{I \times k}$ is a $I \times k$ input matrix of $k$ distinct inputs $x^{(r)} \in \R^I$, $\textbf{1}$ denotes a vector of dimension $k \times 1$ of $1$'s, $\textbf{x}_{j}$ denotes the $j$-th row of the input matrix and $\phi\bullet\textbf{X}$ is the element-wise application of $\phi$ to the matrix $\textbf{X}$. Let $f^{(l)}_{r,i}({\textbf{X},n})=\textbf{1}^{T}_r f^{(l)}_i(\textbf{X},n)=f_i^{(l)}(x^{(r)},n)$ denote the $r$-th component of the $k \times 1$ vector $f^{(l)}_i({\textbf{X},n})$, being $\textbf{1}_r$ a vector of dimension $k \times 1$ with $1$ in the $r$-the entry and $0$ elsewhere.
\end{definition}
\textbf{Remark}: in contrast to \eqref{neural_start}, we have defined \eqref{f1}-\eqref{f2} over an infinite number of units $i \geq 1$ over each layer $l$, but the dependency on each previous layer $l-1$ remains limited to the first $n$ components.

\begin{definition}\label{def:2}
For any fixed $l \geq 2$, let $(\textbf{F}^{(l)}(n))_{n\geq 1}$ be a sequence of stochastic processes on $\mathbb{R}^{I}$. That is, $\textbf{F}^{(l)}(n): \R^I \rightarrow  \R^{\infty}$, with $x \mapsto \textbf{F}^{(l)}(x,n)$, is a stochastic process on $\R^{I}$ whose finite-dimensional distributions are the laws, for any $k\in \N$ and $x^{(1)},\ldots,x^{(k)}\in\R^{I}$, of the $k$-dimensional random vectors
\[
\begin{sistema}
\textbf{F}^{(1)}(\textbf{X}) = \Big[ f_1^{(1)}(\textbf{X}), f_2^{(1)}(\textbf{X}), \dots \Big]^T\\
\textbf{F}^{(l)}(\textbf{X},n) = \Big[ f_1^{(l)}(\textbf{X},n), f_2^{(l)}(\textbf{X},n), \dots \Big]^T.
\end{sistema}
\] 
\end{definition}
\textbf{Remark}: for $k$ inputs, the vector $\textbf{F}^{(l)}(\textbf{X},n)$ is an $\infty \times k$ array, and for a single input $x^{(r)}$, $\textbf{F}^{(l)}(x^{(r)},n)$ can be written as $[f^{(l)}_1(x^{(r)},n), f^{(l)}_2(x^{(r)},n), \dots]^T \in \R^{\infty \times 1}$. We define $\textbf{F}_{r}^{(l)}(\textbf{X},n)=\textbf{F}^{(l)}(x^{(r)},n)$ the $r$-th column of $\textbf{F}^{(l)}(\textbf{X},n)$. When we write $\langle \textbf{F}^{(l-1)}(x,n) ,\textbf{F}^{(l-1)}(y,n) \rangle_{\R^{n}}$ (see (\ref{f1k})) we treat $\textbf{F}^{(l)}(x,n)$ and $\textbf{F}^{(l)}(y,n)$ as elements in $\R^n$ and not in $\R^{\infty}$, i.e. we consider only the first $n$ components of $\textbf{F}^{(l)}(x,n)$ and $\textbf{F}^{(l)}(y,n)$.

\section{Plan sketch}\label{sec:sketch}
We start by recalling the notion of convergence in law, also referred to as convergence in distribution or weak convergence, for a sequence of stochastic processes. See \cite{billingsley1995probability} for a comprehensive account.

\begin{definition}[convergence in distribution]\label{def:3}
Suppose that $f$ and $(f(n))_{n \geq 1}$ are random elements in a topological space $C$. Then, $(f(n))_{n \geq 1}$ is said to converge in distribution to $f$, if $\mathbb{E}[h(f(n))] \rightarrow \mathbb{E}[h(f)]$ as $n \rightarrow \infty$ for every bounded and continuous function $h: C \rightarrow \mathbb{R}$. In that case we write $f(n)\stackrel{d}{\rightarrow}f$.
\end{definition}

In this paper, we deal with continuous and real-valued stochastic processes. More precisely, we consider random elements defined on $C(\R^I;S)$, with $(S,d)$ Polish space. Our aim is to study in $C(\R^I;S)$ the convergence in distribution as the width $n$ goes to infinity for:
\begin{enumerate}[i)]
  \item the sequence $(f_i^{(l)}(n))_{n\geq 1}$ for a fixed $l\geq 2$ and $i \geq 1$ with $(S,d)=(\R, |\cdot|)$, i.e. the neural network process for a single unit;
\item the sequence $(\textbf{F}^{(l)}(n))_{n\geq 1}$ for a fixed $l\geq 2$ with $(S,d)=(\R^{\infty},\| \cdot \|_{\infty})$, i.e. the neural network process for all units.
\end{enumerate}
Since applying Definition \ref{def:3} in a function space is not easy, we need, proved in \ztitleref{sec:appF}, the following proposition.

\begin{prop}[convergence in distribution in $C(\R^I; S)$, $(S,d)$ Polish] \label{Step1}
Suppose that $f$ and $(f(n))_{n \geq 1}$ are random elements in $C(\R^I; S)$ with $(S,d)$ Polish space. Then, $f(n)\stackrel{d}{\rightarrow}f$ if: i) $f(n)\stackrel{f_d}{\rightarrow}f$ and ii) the sequence $(f(n))_{n \geq 1}$ is uniformly tight.
\end{prop}

We denoted with $\stackrel{f_d}{\rightarrow}$ the convergence in law of the finite-dimensional distributions of a sequence of stochastic processes. The notion of tightness formalizes the concept that the probability mass is not allowed to ``escape at infinity'': a single random element $f$ in a topological space $C$ is said to be tight if for each $\epsilon>0$ there exists a compact $T \subset C$ such that $\mathbb{P}[f \in C \setminus T] < \epsilon$. If a metric space $(C,\rho)$ is Polish any random element on the Borel $\sigma$-algebra of $C$ is tight. A sequence of random elements $(f(n))_{n \geq 1}$ in a topological space $C$ is said to be uniformly tight\footnote{\cite{kallenberg2006foundations} uses the same term ``tightness'' for both cases of a single random element and of sequences of random elements; we find that the introduction of ``uniform tightness'' brings more clarity.} if for every $\epsilon>0$ there exists a compact $T \subset C$ such that $\mathbb{P}[f(n) \in C \setminus T] < \epsilon$ for all $n$.

According to Proposition \ref{Step1}, to achieve convergence in distribution in function spaces we need the following Steps A-D:

\textbf{Step A)} to establish the existence of the finite-dimensional weak-limit $f$ on $\R^I$. We will rely on Theorem 5.3 of \cite{kallenberg2006foundations}, known as Levy theorem.

\textbf{Step B)} to establish the existence of the stochastic processes $f$ and $(f(n))_{n \geq 1}$ as elements in $S^{\R^I}$ the space of function from $\R^I$ into $S$. We make use of Daniell-Kolmogorov criterion \cite[Theorem 6.16]{kallenberg2006foundations}: given a family of multivariate distributions $\{P_{\mathcal{I}} \text{ probability measure on } \R^{\dim({\mathcal{I}})} \mid {\mathcal{I}} \subset \{x^{(1)}, \dots, x^{(k)}\}_{x^{(z)} \in \R^I, k \in \N}\}$ there exists a stochastic process with $\{P_{\mathcal{I}}\}$ as finite-dimensional distributions if $\{P_{\mathcal{I}}\}$ satisfies the projective property: $P_J(\cdot \times \R_{J \setminus {\mathcal{I}}}) = P_{\mathcal{I}}(\cdot), {\mathcal{I}} \subset J \subset \{x^{(1)}, \dots, x^{(k)}\}_{x^{(z)} \in \R^I, k \in \N}$. That is, it is required consistency with respect to the marginalization over arbitrary components. In this step we also suppose, for a moment, that the stochastic processes $(f(n))_{n \geq 1}$ and $f$ belong to $C(\R^I;S)$ and we establish the existence of such stochastic processes in $C(\R^I;S)$ endowed with a $\sigma$-algebra and a probability measure that will be defined.

\textbf{Step C)} to show that the stochastic processes $(f(n))_{n \geq 1}$ and $f$ belong to $C(\R^I;S)\subset S^{\R^I}$. With regards to $(f(n))_{n \geq 1}$ this is a direct consequence of (\ref{f1})-(\ref{f2}) and the continuity of $\phi$. With regards to the limiting process $f$, with an additional Lipschitz assumption on $\phi$, we rely on the following Kolmogorov-Chentsov criterion \cite[Theorem 3.23]{kallenberg2006foundations}:
    
    \begin{prop}[continuous version and local-H\"olderianity, $(S,d)$ complete] \label{Kolmogorov_Chentsov_0}
    Let $f$ be a process on $\R^I$ with values in a complete metric space $(S, d)$, and assume that there exist $a,b,H > 0$ such that,
    \begin{equation*}
    \E[d(f(x),f(y))^a] \leq H{\|x - y\|}^{(I + b)},\quad x,y \in \R^I
    \end{equation*}
    Then $f$ has a continuous version (i.e. $f$ belongs to $C(\R^I;S)$), and the latter is a.s. locally H\"older continuous with exponent $c$ for any $c \in (0, \nicefrac{b}{a})$.
    \end{prop}
    
\textbf{Step D)} the uniform tightness of $(f(n))_{n\geq 1}$ in $C(\R^I;S)$. We rely on an extension of the Kolmogorov-Chentsov criterion \cite[Corollary 16.9]{kallenberg2006foundations}, which is stated in the following proposition.

\begin{prop}[uniform tightness in $C(\R^I; S)$, $(S,d)$ Polish] \label{Kolmogorov_Chentsov}
Suppose that $(f(n))_{n \geq 1}$ are random elements in $C(\R^I;S)$ with $(S,d)$ Polish space. Assume that $f(0_{\R^I},n)_{n \geq 1}$ (i.e. $f(n)$ evaluated at the origin) is uniformly tight in $S$ and that there exist $a,b,H > 0$ such that,
\begin{equation*}
\E[d(f(x,n), f(y,n))^a] \leq H{\|x - y\|}^{(I + b)},\quad x,y \in \R^I, n \in \N
\end{equation*}
uniformly in $n$. Then $(f(n))_{n \geq 1}$ is uniformly tight in $C(\R^I;S)$.
\end{prop}

\section{Large-width functional limits}\label{sec:limit}

\subsection{Limit on $C(\R^I;S)$, with $(S,d)=(\R,| \cdot |)$, for a fixed unit $i\geq 1$ and layer $l$}\label{sec:limitR}
\begin{lemma}[finite-dimensional limit]\label{Lemma1}
If $\phi$ satisfies ($\ref{envelope}$) then there exists a stochastic process $f^{(l)}_i : \R^I \to \R$ such that $( f^{(l)}_i (n) )_{n \geq 1}\stackrel{f_d}{\rightarrow} f^{(l)}_i$ as $n \rightarrow \infty$.
\end{lemma}
\begin{proof}
Fix $l \geq 2$ and $i\geq 1$. Fixed $k$ inputs $\textbf{X}=[x^{(1)}, \dots , x^{(k)}]$, we show that as $n \rightarrow +\infty$
\begin{align}\label{thm2}
f^{(l)}_i(\textbf{X},n) \stackrel{d}{\rightarrow} N_k(\textbf{0},\Sigma(l)),
\end{align}
where $\Sigma(l)$ denotes the $k \times k$ covariance matrix, which can be computed through the recursion: 
$\Sigma(1)_{i,j} = \sigma_{b}^2+\sigma_{\omega}^2 \langle x^{(i)} , x^{(j)} \rangle_{\R^I}$, $
\Sigma(l)_{i,j}= \sigma_b^2+\sigma_{\omega}^2 \int\phi( f_i) \phi (f_j) q^{(l-1)}(\dd f)$, where $q^{(l-1)}=N_k(\textbf{0},\Sigma(l-1))$. By means of (\ref{caromega}), (\ref{carbias}), (\ref{f1}) and (\ref{f2}),
\begin{align}\label{f1k}
\begin{sistema}
f^{(1)}_{i}(\textbf{X}) \stackrel{d}{=}N_k(\textbf{0},\Sigma(1)), \quad \Sigma(1)_{i,j}=\sigma_{b}^2+\sigma_{\omega}^2 \langle x^{(i)}, x^{(j)} \rangle_{\R^I}\\
f^{(l)}_i({\textbf{X},n}) | f^{(l-1)}_{1,\dots ,n} \stackrel{d}{=}N_k(\textbf{0},\Sigma(l,n)), \quad \text{ for } l \geq 2,\\
\quad \Sigma(l,n)_{i,j}=\sigma_{b}^2+ \frac{\sigma_{\omega}^2}{n} \Big\langle  ( \phi \bullet \textbf{F}^{(l-1)}_{i}({\textbf{X},n}) ),  ( \phi \bullet \textbf{F}^{(l-1)}_{j}({\textbf{X},n}) ) \Big\rangle_{\R^n}
\end{sistema}
\end{align}
We prove (\ref{thm2}) using Levy's theorem, that is the point-wise convergence of the sequence of characteristic functions of (\ref{f1k}). We defer to \ztitleref{sec:appA} for the complete proof.
\end{proof}

Lemma \ref{Lemma1} proves \textbf{Step A}. This proof gives an alternative and self-contained proof of the main result of \cite{matthews2018gaussianB}, under the more general assumption that the activation function $\phi$ satisfies the polynomial envelop \eqref{envelope}. Now we prove \textbf{Step B}, i.e. the existence of the stochastic processes $f^{(l)}_i (n)$ and $f^{(l)}_i$ on the space $\R^{\R^I}$, for each layer $l\geq 1$, unit $i\geq 1$ and $n \in \N$. In \ztitleref{sec:appE}.1 we show that the finite-dimensional distributions of $f^{(l)}_i (n)$ satisfies Daniell-Kolmogorov criterion \cite[Theorem 6.16]{kallenberg2006foundations}, and hence the stochastic process $f^{(l)}_i (n)$ exists. In \ztitleref{sec:appE}.2 we prove a similar result for the finite-dimensional distributions of the limiting process $f_i^{(l)}$. In \ztitleref{sec:appE}.3 we prove that, if these stochastic processes are continuous, they are naturally defined in $C(\R^I;\R)$. In order to prove the continuity, i.e. \textbf{Step C} note that $f^{(1)}_i(x)=\sum_{j=1}^I \omega_{i,j}^{(1)}x_j+b^{(1)}_i$ is continuous by construction, thus by induction on $l$, if $f^{(l-1)}_i(n)$ are continuous for each $i\geq 1$ and $n$, then $f^{(l)}_i(x,n)=\frac{1}{\sqrt{n}}\sum_{j=1}^n \omega_{i,j}^{(l)}\phi(f^{(l-1)}_j(x,n))+b^{(l)}_i$ is continuous being composition of continuous functions. For the limiting process $f^{(l)}_i$ we assume $\phi$ to be Lipschitz with Lipschitz constant $L_{\phi}$. In particular we have the following:

\begin{lemma}[continuity]\label{Lemma2}
If $\phi$ is Lipschitz on $\R$ then $f^{(l)}_i(1),f^{(l)}_i(2), \dots$ are $\mathbb{P}$-a.s. Lipschitz on $\R^I$, while the limiting process $f^{(l)}_i$ is $\mathbb{P}$-a.s. continuous on $\R^I$ and locally $\gamma$-H\"older continuous for each $0< \gamma <1$.
\end{lemma}

\begin{proof}
Here we present a sketch of the proof, and we defer to \ztitleref{sec:appB1} and \ztitleref{sec:appB2} for the complete proof. For $(f^{(l)}_i(n))_{n \geq 1}$ it is trivial to show that for each $n$
\begin{align}\label{Lipschitz}
|f^{(l)}_i(x,n)-f^{(l)}_i(y,n)|  \leq H^{(l)}_i(n)\| x-y \|_{\R^I}, \quad x,y \in \R^I,\mathbb{P}-a.s.
\end{align}
where $H^{(l)}_i(n)$ denotes a suitable random variable, which is defined by the following recursion over $l$
\begin{align}\label{Lipsch_i_n}
\begin{sistema}
H^{(1)}_i(n)=\sum_{j=1}^{I} \big| \omega_{i,j}^{(1)}\big| \\
H^{(l)}_i(n)=\frac{L_{\phi}}{\sqrt{n}} \sum_{j=1}^{n} \big| \omega_{i,j}^{(l)}\big|  H^{(l-1)}_j(n)
\end{sistema}
\end{align}
To establish the continuity of the limiting process $f^{(l)}_i$ we rely on Proposition \ref{Kolmogorov_Chentsov_0}. Take two inputs $x,y \in \R^I$. From (\ref{thm2}) we get that $[f^{(l)}_i(x),f^{(l)}_i(y)] \sim N_2(\textbf{0}, \Sigma(l))$ where
\[
 \begin{split}
& \Sigma(1)=
 \sigma_b^2
\begin{bmatrix}
 1 & 1 \\
 1 & 1 \\
 \end{bmatrix}
 +  \sigma_{\omega}^2 
\begin{bmatrix} 
\| x \|_{\R^I}^2 & \langle x , y \rangle_{\R^I} \\
\langle x , y \rangle_{\R^I} & \| y \|_{\R^I}^2 \\
\end{bmatrix},\\
 &\Sigma(l)= \sigma_b^2
\begin{bmatrix}
 1 & 1 \\
 1 & 1 \\
 \end{bmatrix}
 +  \sigma_{\omega}^2  \int
\begin{bmatrix} 
| \phi(u) |^2 &  \phi(u) \phi(v) \\
\phi(u) \phi(v) & | \phi(v) |^2 \\
\end{bmatrix}
q^{(l-1)} (\dd u, \dd v),
\end{split}
\]
where $q^{(l-1)}= N_2(\textbf{0}, \Sigma (l-1))$. Defining $\textbf{a}^T=[1,-1]$, from (\ref{thm2}) we know that  $f^{(l)}_i(y)-f^{(l)}_i(x) \sim N(\textbf{a}^T\textbf{0},\textbf{a}^T \Sigma(l) \textbf{a})$.Thus
\[
|f^{(l)}_i(y)-f^{(l)}_i(x)|^{2\theta} \sim | \sqrt{\textbf{a}^T\Sigma(l) \textbf{a}} N(0,1) |^{2\theta} \sim (\textbf{a}^T\Sigma(l) \textbf{a})^{\theta}  |N(0,1)|^{2\theta}.
\]
We proceed by induction over the layers. For $l=1$,
\[
\begin{split}
\E \Big[| f^{(1)}_i(y)-f^{(1)}_i(x) |^{2\theta} \Big] &= C_{\theta} (\textbf{a}^T\Sigma(1) \textbf{a})^{\theta}\\
&=C_{\theta} (\sigma_{\omega}^2 \|y \|_{\R^I}^{2} -2\sigma_{\omega}^2 \langle y,x \rangle_{\R^I} +\sigma_{\omega}^2 \| x \|_{\R^I}^2 )^{\theta}\\
&=C_{\theta} (\sigma_{\omega}^2)^{\theta} ( \|y \|_{\R^I}^{2} -2\langle y,x \rangle_{\R^I} + \| x \|_{\R^I}^2 )^{\theta}\\
&= C_{\theta} (\sigma_{\omega}^2)^{\theta} \|y-x \|_{\R^I}^{2\theta},
\end{split}
\]
where $C_{\theta}= \E[|N(0,1)|^{2\theta}]$. By hypothesis induction there exists a constant $H^{(l-1)}>0$ such that $\int |u-v|^{2\theta} q^{(l-1)}(\dd u, \dd v) \leq H^{(l-1)} \|y-x \|^{2\theta}_{\R^I}$. Then,
\[
\begin{split}
| f^{(l)}_i(y)-f^{(l)}_i(x) |^{2\theta} &\sim |N(0,1)|^{2\theta} (\textbf{a}^T\Sigma(l) \textbf{a})^{\theta}\\
& = |N(0,1)|^{2\theta} \Big(  \sigma^2_{\omega}\int  [ |\phi(u)|^2-2\phi(u)\phi(v)+|\phi(v)|^2 ] q^{(l-1)}(\dd u, \dd v) \Big)^{\theta}\\
& \leq |N(0,1)|^{2\theta} (\sigma^2_{\omega} L_{\phi}^2)^{\theta} \int |u-v|^{2\theta} q^{(l-1)}(\dd u, \dd v)\\
& \leq |N(0,1)|^{2\theta} (\sigma^2_{\omega} L_{\phi}^2)^{\theta} H^{(l-1)}\| y-x \|_{\R^{I}}^{2\theta}.
\end{split}
\]
where we used $|\phi(u)|^2-2\phi(u)\phi(v)+|\phi(v)|^2 = | \phi(u)-\phi(v) |^2 \leq L_{\phi}^2 |u-v|^2$ and the Jensen inequality. Thus,
\begin{align} \label{expec}
\begin{split}
\E \Big[| f^{(l)}_i(y)-f^{(l)}_i(x) |^{2\theta} \Big] \leq H^{(l)}\|y - x \|_{\R^I}^{2\theta},
\end{split}
\end{align}
where the constant $H^{(l)}$ can be explicitly derived by solving the following system
\begin{align}\label{H^l}
\begin{sistema}
H^{(1)} = C_{\theta} (\sigma^2_{\omega})^{\theta} \\
H^{(l)} = C_{\theta} (\sigma^2_{\omega} L_{\phi}^2)^{\theta} H^{(l-1)}.
\end{sistema}
\end{align}
It is easy to get $H^{(l)}= C_{\theta}^l(\sigma^2_{\omega})^{l\theta} (L_{\phi}^2)^{(l-1)\theta}$. Observe that $H^{(l)}$ does not depend on $i$ (this will be helpful in establishing the uniformly tightness of $(f^{(l)}_i(n))_{n \geq 1}$ and the continuity of $\mathbf{F}^{(l)}$). By Proposition \ref{Kolmogorov_Chentsov_0}, setting $\alpha=2\theta$, and $\beta=2\theta-I$ (since $\beta$ needs to be positive, it is sufficient to choose $\theta>I/2$) we get that $f^{(l)}_i$ has a continuous version and the latter is $\mathbb{P}$-a.s locally $\gamma$-H\"older continuous for every $0<\gamma<1-\frac{I}{2\theta}$, for each $\theta>I/2$. Taking the limit as $\theta \rightarrow +\infty$ we conclude the proof.
\end{proof}

\begin{lemma}[uniform tightness]\label{Lemma3}
If $\phi$ is Lipschitz on $\R$ then $( f^{(l)}_i(n) )_{n\geq 1}$ is uniformly tight in $C(\R^I;\R)$.
\end{lemma}

\begin{proof}
We defer to \ztitleref{sec:appB3} for details. Fix $i \geq 1$, $l\geq1$. We apply Proposition \ref{Kolmogorov_Chentsov} to show the uniform tightness of the sequence $(f^{(l)}_i(n))_{n\geq 1}$ in $C(\R^I;\R)$. By Lemma \ref{Lemma2} $f^{(l)}_i(1),f^{(l)}_i(2), \dots$ are random elements in $C(\R^I;\R)$. 
Since $(\R,|\cdot|)$ is Polish, every probability measure is tight, then $f(0_{\R^I},n)$ is tight in $\R$ for every $n$. Moreover, by Lemma \ref{Lemma1} $f_i(0_{\R^I},n)_{n \geq 1}\stackrel{d}{\rightarrow}f_i^{(l)}(0_{\R^I})$, therefore by \cite[Theorem 11.5.3]{dudley2002real}, $f(0_{\R^I},n)_{n \geq 1}$ is uniformly tight in $\R$.

It remains to show that there exist two values $\alpha >0$ and $\beta >0$, and a constant $H^{(l)}>0$ such that
\[
\E \Big[| f_{i}^{(l)}(y,n)-f_{i}^{(l)}(x,n) |^{\alpha} \Big] \leq H^{(l)} \|y-x \|_{\R^I}^{I+\beta}, \quad x,y \in \R^I, n \in \N
\]
uniformly in $n$. Take two points $x,y \in \R^I$. From (\ref{f1k}) we know that $f_{i}^{(l)}(y,n)|f^{(l-1)}_{1,\dots ,n} \sim N(0,\sigma^2_y(l,n))$ and $f_{i}^{(l)}(x,n)|f^{(l-1)}_{1,\dots ,n} \sim N(0, \sigma^2_x(l,n))$ with joint distribution $N_2(\textbf{0}, \Sigma(l,n))$, where
\[
\Sigma(1)=
\begin{bmatrix}
  \sigma^2_x(1) & \Sigma(1)_{x,y} \\
  \Sigma(1)_{x,y} & \sigma^2_y(1) \\
  \end{bmatrix},
  \quad \Sigma(l)= 
\begin{bmatrix}
  \sigma^2_x(l,n) & \Sigma(l,n)_{x,y} \\
  \Sigma(l,n)_{x,y} & \sigma^2_y(l,n) \\
  \end{bmatrix},
\]
with,
\[
\begin{sistema}
\sigma^2_x(1) = \sigma_{b}^2 + \sigma_{\omega}^2 \| x \|_{\R^I}^2,\\
\sigma^2_y(1) = \sigma_{b}^2 + \sigma_{\omega}^2 \| y \|_{\R^I}^2,\\
\Sigma(1)_{x,y} = \sigma_{b}^2 + \sigma_{\omega}^2 \langle x , y \rangle_{\R^I},\\
\sigma^2_x(l,n) = \sigma_{b}^2 + \frac{\sigma_{\omega}^2}{n} \sum_{j=1}^n | \phi \circ f_j^{(l-1)} (x,n) |^2, \\
\sigma^2_y(l,n) = \sigma_{b}^2 + \frac{\sigma_{\omega}^2}{n} \sum_{j=1}^n | \phi \circ f_j^{(l-1)} (y,n) |^2, \\
\Sigma(l,n)_{x,y} = \sigma_{b}^2 + \frac{\sigma_{\omega}^2}{n} \sum_{j=1}^n \phi( f^{(l-1)}_j(x,n)) \phi(f^{(l-1)}_j(y,n))
\end{sistema}
\]
Defining $\textbf{a}^T=[1,-1]$ we have that $f_{i}^{(l)}(y,n)|f^{(l-1)}_{1,\dots ,n}-f_{i}^{(l)}(x,n)|f^{(l-1)}_{1,\dots ,n}$ is distributed as $N(\textbf{a}^T\textbf{0},\textbf{a}^T \Sigma(l,n) \textbf{a})$, where $\textbf{a}^T \Sigma(l,n) \textbf{a} = \sigma_y^2(l,n)-2\Sigma(l,n)_{x,y}+\sigma_x^2(l,n)$. Consider $\alpha=2\theta$ with $\theta$ integer. Thus
\[
\Big|f_{i}^{(l)}(y,n)|f^{(l-1)}_{1,\dots ,n}-f_{i}^{(l)}(x,n)|f^{(l-1)}_{1,\dots ,n}\Big|^{2\theta} \sim | \sqrt{\textbf{a}^T\Sigma(l,n) \textbf{a}} N(0,1) |^{2\theta} \sim (\textbf{a}^T\Sigma(l,n) \textbf{a})^{\theta}  |N(0,1)| ^{2\theta}.
\]
As in previous theorem, for $l=1$ we get
$\E \Big[| f_i^{(1)}(y,n)-f_i^{(1)}(x,n)|^{2\theta} \Big] = C_{\theta}(\sigma_{\omega}^2)^{\theta} \|y-x \|_{\R^I}^{2\theta}$
where $C_{\theta}=\E[|N(0,1)^{2\theta}|]$. Set $H^{(1)}=C_{\theta} (\sigma_{\omega}^2)^{\theta}$ and by hypothesis induction suppose that for every $j\geq 1$
\[
\E \Big[| f^{(l-1)}_j (y,n)- f^{(l-1)}_j (x,n) |^{2\theta}\Big] \leq H^{(l-1)} \| y-x \|_{\R^I}^{2\theta}.
\]
By hypothesis $\phi$ is Lipschitz, then
\[
\begin{split}
\E\Big[ | f_{i}^{(l)}(y,n)-f_{i}^{(l)}(x,n) |^{2\theta} \Big|f^{(l-1)}_{1,\dots ,n} \Big] & = C_{\theta} (\textbf{a}^T\Sigma(l,n) \textbf{a})^{\theta}\\
&=C_{\theta}  \Big( \sigma_y^2(l,n)-2\Sigma(l,n)_{x,y}+\sigma_x^2(l,n)  \Big)^{\theta}\\
&= C_{\theta}  \Big( \frac{\sigma^2_{\omega}}{n} \sum_{j=1}^n \Big|  \phi \circ f_j^{(l-1)} (y,n) - \phi \circ f_j^{(l-1)} (x,n) \Big|^2  \Big)^{\theta} \\
&\leq C_{\theta}  \Big( \frac{\sigma^2_{\omega}L_{\phi}^2}{n} \sum_{j=1}^n\Big|  f_j^{(l-1)} (y,n)  - f_j^{(l-1)} (x,n)  \Big|^2 \Big)^{\theta}\\
&= C_{\theta}  \frac{(\sigma^2_{\omega}L_{\phi}^2)^{\theta}}{n^{\theta}} \Big( \sum_{j=1}^n\Big|  f_j^{(l-1)} (y,n)  - f_j^{(l-1)} (x,n)  \Big|^2 \Big)^{\theta}\\
&\leq C_{\theta}  \frac{(\sigma^2_{\omega}L_{\phi}^2)^{\theta}}{n} \sum_{j=1}^n\Big|  f_j^{(l-1)} (y,n)  - f_j^{(l-1)} (x,n)  \Big|^{2\theta}.\\
\end{split}
\]
Using the induction hypothesis 
\[
\begin{split}
\E \Big[ | f_{i}^{(l)}(y,n)-f_{i}^{(l)}(x,n) |^{2\theta} \Big] &=\E \Big[ \E \Big[ | f_{i}^{(l)}(y,n)-f_{i}^{(l)}(x,n) |^{2\theta} \Big|f^{(l-1)}_{1,\dots ,n} \Big]\Big] \\
& \leq C_{\theta} \frac{(\sigma^2_{\omega} L_{\phi}^2)^{\theta}}{n} \sum_{j=1}^n \E \Big[ |  f_j^{(l-1)} (y,n)  - f_j^{(l-1)} (x,n) |^{2\theta} \Big] \\
& \leq C_{\theta} (\sigma^2_{\omega} L_{\phi}^2)^{\theta} H^{(l-1)} \| y-x \|^{2\theta}_{\R^I}.
\end{split}
\]
We can get the constant $H^{(l)}$ by solving the same system as (\ref{H^l}), obtaining $H^{(l)}= C_{\theta}^l(\sigma^2_{\omega})^{l\theta} (L_{\phi}^2)^{(l-1)\theta}$ which does not depend on $n$. By Proposition \ref{Kolmogorov_Chentsov} setting $\alpha=2\theta$ and $\beta= 2\theta-I$, since $\beta$ must be a positive constant, it is sufficient to take $\theta >I/2$ and this concludes the proof.
\end{proof}
Note that Lemma \ref{Lemma3} provides the last \textbf{Step D} that allows us to prove the desired result which is explained in the theorem that follows:
\begin{theorem}[functional limit]\label{THM1}
If $\phi$ is Lipschitz on $\R$ then $f^{(l)}_i(n) \stackrel{d}{\rightarrow} f^{(l)}_i$ on $C(\R^I;\R)$.
\end{theorem}

\begin{proof}
We apply Proposition \ref{Step1} to $(f_i^{(l)}(n))_{n \geq 1}$. By Lemma \ref{Lemma2}, we have that $f^{(l)}_i,( f^{(l)}_i (n) )_{n \geq 1}$ belong to $C(\R^I;\R)$. From Lemma \ref{Lemma1} we have the convergence of the finite-dimensional distributions of $(f_i^{(l)}(n))_{n \geq 1}$, and form Lemma \ref{Lemma3} we have the uniform tightness of $(f_i^{(l)}(n))_{n \geq 1}$.
\end{proof}

\subsection{Limit on $C(\R^I;S)$, with $(S,d)=(\R^{\infty},\| \cdot \|_{\R^{\infty}})$, for a fixed layer $l$}\label{sec:limitRinfinit}

As in the previous section we prove Steps A-D for the sequence $(\textbf{F}^{(l)}(n))_{n\geq 1}$. Remark that each stochastic process $\textbf{F}^{(l)},\textbf{F}^{(l)}(1), \textbf{F}^{(l)}(2), \dots$ defines on $C(\R^I;\R^{\infty})$ a joint measure whose $i$-th marginal is the measure induced respectively by $f^{(l)}_i,f^{(l)}_i(n),f^{(2)}_i(n), \dots$ (see \ztitleref{sec:appE}.1 -\ztitleref{sec:appE}.4). Let $\textbf{F}^{(l)}\stackrel{d}{=}\bigotimes_{i=1}^{\infty}f_i^{(l)}$, where $\bigotimes$ denotes the product measure.

\begin{lemma}[finite-dimensional limit]\label{Lemma4} If $\phi$ satisfies (\ref{envelope})
 then $\textbf{F}^{(l)}(n) \stackrel{f_d}{\rightarrow} \textbf{F}^{(l)}$ as $n \rightarrow \infty$.
\end{lemma}

\begin{proof}
The proof follows by Lemma \ref{Lemma1} and Cram\'er-Wold theorem for finite-dimensional projection of $\textbf{F}^{(l)}(n)$: it is sufficient to establish the large $n$ asymptotic of linear combinations of the $f^{(l)}_i(\textbf{X},n)$'s for $i \in \mathcal{L} \subset \N$. In particular, we show that for any choice of inputs elements $\textbf{X}$, as $n \rightarrow +\infty$
\begin{align}
\label{product}
\textbf{F}^{(l)}(\textbf{X},n) \stackrel{d}{\rightarrow} \bigotimes_{i=1}^{\infty} N_k(\textbf{0},\Sigma(l)),
\end{align}
where $\Sigma(l)$ is defined in (\ref{thm2}). The proof is reported in \ztitleref{sec:appC}.
\end{proof}

\begin{lemma}[continuity]\label{Lemma5}
If $\phi$ is Lipschitz on $\R$ then $\textbf{F}^{(l)},( \textbf{F}^{(l)} (n) )_{n \geq 1}$ belong to $C(\R^I;\R^{\infty})$. More precisely $\textbf{F}^{(l)}(1),\textbf{F}^{(l)}(2), \dots$ are $\mathbb{P}$-a.s. Lipschitz on $\R^I$, while the limiting process $\textbf{F}^{(l)}$ is $\mathbb{P}$-a.s. continuous on $\R^I$ and locally $\gamma$-H\"older continuous for each $0< \gamma <1$.
\end{lemma}

\begin{proof}
It derives immediately from Lemma \ref{Lemma2}. We defer to \ztitleref{sec:appD1} and \ztitleref{sec:appD2} for details. The continuity of the sequence process immediately follows from the Lipschitzianity of each component in (\ref{Lipschitz}) while the continuity of the limiting process $\textbf{F}^{(l)}$ is proved by applying Proposition \ref{Kolmogorov_Chentsov_0}. Take two inputs $x,y \in \R^I$ and fix $\alpha \geq 1$ even integer. Since $\xi(t)\leq t$ for all $t \geq 0$, and by Jensen inequality 
\[
d\big(\textbf{F}^{(l)}(x),\textbf{F}^{(l)}(y)\big)_{\infty} ^{\alpha} \leq \Big( \sum_{i=1}^{\infty}\frac{1}{2^i}|f_i^{(l)}(x)-f^{(l)}_i(y)|\Big)^{\alpha} \leq \sum_{i=1}^{\infty}\frac{1}{2^i} |f_i^{(l)}(x)-f^{(l)}_i(y)|^{\alpha} 
\]

Thus, by applying monotone convergence theorem to the positive increasing sequence $g(N)=\sum_{i=1}^{N}\frac{1}{2^i} |f_i^{(l)}(x)-f^{(l)}_i(y)|^{\alpha}$ (which allows to exchange $\E$ and $\sum_{i=1}^{\infty}$), we get

\[
\begin{split}
\E \Big[ d\big(\textbf{F}^{(l)}(x),&\textbf{F}^{(l)}(y)\big)_{\infty} ^{\alpha} \Big] \leq \E \Big[ \sum_{i=1}^{\infty}\frac{1}{2^i} |f^{(l)}(x)-f^{(l)}_i(y)|^{\alpha} \Big] = \lim_{N \rightarrow \infty} \E \Big[ \sum_{i=1}^{N}\frac{1}{2^i} |f_i^{(l)}(x)-f^{(l)}_i(y)|^{\alpha} \Big] \\
& = \sum_{i=1}^{\infty}\frac{1}{2^i} \E \Big[ |f_i^{(l)}(x)-f^{(l)}_i(y)|^{\alpha} \Big] = \sum_{i=1}^{\infty}\frac{1}{2^i} H^{(l)} \|x - y \|^{\alpha}_{\R^I} = H^{(l)} \|x - y \|_{\R^I}^{\alpha} 
\end{split}
\]
where we used (\ref{expec}) and the fact that $H^{(l)}$ does not depend on $i$ (see (\ref{H^l})). Therefore, by Proposition \ref{Kolmogorov_Chentsov_0}, for each $\alpha > I$, setting $\beta=\alpha-I$ (since $\beta$ needs to be positive, it is sufficient to choose $\alpha>I$) $\textbf{F}^{(l)}$ has a continuous version $\textbf{F}^{(l)(\theta)}$ which is $\mathbb{P}$-a.s locally $\gamma$-H\"older continuous for every $0<\gamma<1-\frac{I}{\alpha}$. Letting $\alpha \rightarrow \infty$ we conclude.
\end{proof}

\begin{theorem}[functional limit]\label{THM2}
If $\phi$ is Lipschitz on $\R$ then $(\textbf{F}^{(l)}(n))_{n\geq1} \stackrel{d}{\rightarrow} \textbf{F}^{(l)}$ as $n \rightarrow \infty$ on $C(\R^I;\R^{\infty})$.
\end{theorem}

\begin{proof}
This is Proposition \ref{Step1} applied to $(\textbf{F}^{(l)}(n))_{n\geq1}$. From Lemma \ref{Lemma4} and Lemma \ref{Lemma5} it remains to show the uniform tightness of the sequence $(\textbf{F}^{(l)}(n))_{n\geq1}$ in $C(\R^I;\R^{\infty})$. 
Let $\epsilon > 0$ and let $(\epsilon_i)_{i\geq1}$ be a positive sequence such that $\sum_{i=1}^{\infty} \epsilon_i = \epsilon/2$ . We have established the uniform tightness of each component (Lemma \ref{Lemma3}). Therefore for each $i \in \N$ there exists a compact $K_i \subset C(\R^I;\R)$ such that $\mathbb{P}[f_i^{(l)}(n) \in C(\R^I;\R) \setminus K_i]< \epsilon_i$ for each $n \in \N$ (such compact depends on $\epsilon_i$). Set $K= \times_{i=1}^{\infty}K_i$ which is compact by Tychonoff theorem. Note that this is a compact on the product space $\times_{i=1}^{\infty}C(\R^I;\R)$ with associated product topology, and this is also a compact on $C(\R^I;\R^{\infty})$ (see \ztitleref{sec:appE}.4). Then $\mathbb{P} \Big[ \textbf{F}^{(l)}(n) \in C(\R^I;\R^{\infty}) \setminus K \Big] = \mathbb{P} \Big[ \bigcup_{i=1}^{\infty}\{ f^{(l)}_i(n) \in C(\R^I;\R) \setminus K_i \} \Big] \leq \sum_{i=1}^{\infty} \mathbb{P} \Big[ f^{(l)}_i(n) \in C(\R^I;\R) \setminus K_i \Big] \leq \sum_{i=1}^{\infty} \epsilon_i < \epsilon
$ which concludes the proof.
\end{proof}

\section{Discussion} \label{sec:discussion}

We looked at deep Gaussian neural networks as stochastic processes, i.e. infinite-dimensional random elements, on the input space $\mathbb{R}^{I}$, and we showed that: i) a network defines a stochastic process on the input space $\mathbb{R}^{I}$; ii) under suitable assumptions on the activation function, a network with re-scaled weights converges weakly to a Gaussian Process in the large-width limit. These results extend previous works \citep{neal1995bayesian,der2006beyond,lee2018deep,matthews2018gaussian,matthews2018gaussianB,yang2019wide} that investigate the limiting distribution of neural network over a countable number of distinct inputs. From the point of view of applications, the convergence in distribution is the starting point for the convergence of expectations. Let consider a continuous function $g: C(\R^I;\R^\infty) \rightarrow \R$. By the continuous function mapping theorem \citep[Theorem 2.7]{billingsley1999convergence}, we have $g(\textbf{F}^{(l)}(n)) \overset{d}{\rightarrow} g(\textbf{F}^{(l)})$ as $n\rightarrow+\infty$, and under uniform integrability \citep[Section 3]{billingsley1999convergence}, we have \citep[Theorem 3.5]{billingsley1999convergence} $\E[g(\textbf{F}^{(l)}(n))] \rightarrow \E[g(\textbf{F}^{(l)})]$ as $n\rightarrow+\infty$. See also \cite{dudley2002real} and references therein.

As a by-product of our results we showed that, under a Lipschitz activation function, the limiting Gaussian Process has almost surely locally $\gamma$-H\"older continuous paths, for $0 < \gamma <1$. This raises the question on whether it is possible to strengthen our results to cover the case $\gamma =1$, or even the case of local Lipschitzianity of the paths of the limiting process. In addition, if the activation function is differentiable, does this property transfer to the limiting process? We leave these questions to future research. Finally, while fully-connected deep neural networks represent an ideal starting point for theoretical analysis, modern neural network architectures are composed of a much richer class of layers which includes convolutional, residual, recurrent and attention components. The technical arguments followed in this paper are amenable to extensions to more complex network architectures. Providing a mathematical formulation of network's architectures and convergence results in a way that it allows for extensions to arbitrary architectures, instead of providing an ad-hoc proof for each specific case, is a fundamental research problem. Greg Yang’s work on Tensor Programs \citep{yang2019wide} constitutes an important step in this direction.

\bibliographystyle{iclr2021_conference}
\bibliography{ref.bib}

\clearpage

\appendix

\section*{SM A} \zlabel{sec:appA}
The following inequalities will be used different times in the proofs without explicit mention:
\begin{itemize}
\item[1)]
For any real values $a_1, \dots, a_n \geq 0$ and $s \geq 1$,
\begin{align*}
(a_1 + \dots + a_n)^s \leq n^{s-1}(a_1^s+ \dots + a_n^s).
\end{align*}
It follows immediately considering the convex function $x \mapsto x^s$ applied to the the weighted sum $\nicefrac{a_1 + \dots + a_n}{n}$.

\item[2)]For every values $a_1, \dots, a_n \in \R$ and $0 < s < 1$,
\begin{align*}
|a_1 + \dots + a_n|^s \leq |a_1|^s+ \dots + |a_n|^s.
\end{align*}

It follows immediately studying the $s$-H\"older function $x \mapsto |x|^s$.
\end{itemize} 

By means of (\ref{caromega}), (\ref{carbias}) and (\ref{f1}), we can write for $i\geq 1$ and $l \geq 2$
\[
\begin{split}
\varphi_{f^{(1)}_i({\textbf{X},n})}(\textbf{t})&=\E [e^{\ii \textbf{t}^Tf^{(1)}_i({\textbf{X},n})}]\\
&=\E \Big[ \exp \Big\{ \ii \textbf{t}^T \Big[ \sum_{j=1}^{I}\omega_{i,j}^{(1)}\textbf{x}_{j}+ b_{i}^{(1)}\textbf{1}\Big] \Big\}\Big]\\
&=\E \Big[ \exp \Big\{ \ii \textbf{t}^T b_{i}^{(1)}\textbf{1} +\ii \textbf{t}^T \sum_{j=1}^{I}\omega_{i,j}^{(1)}\textbf{x}_{j} \Big\}\Big]\\
&=\E \Big[ \exp \Big\{ \ii (\textbf{t}^T \textbf{1}) b_{i}^{(1)} \Big\}\Big]\prod_{j=1}^{I} \E \Big[ \exp \Big\{ \ii (\textbf{t}^T \textbf{x}_{j})\omega_{i,j}^{(1)} \Big\}\Big]\\
&= \exp \Big\{-\frac{1}{2} \sigma_{b}^2(\textbf{t}^T \textbf{1})^2 \Big\} \prod_{j=1}^{I} \exp \Big\{ -\frac{1}{2} \sigma_{\omega}^2 (\textbf{t}^T \textbf{x}_{j})^2 \Big\}\\
&= \exp \Big\{-\frac{1}{2} \Big[ \sigma_{b}^2(\textbf{t}^T \textbf{1})^2+ \sigma_{\omega}^2 \sum_{j=1}^{I} (\textbf{t}^T \textbf{x}_{j})^2 \Big]\Big\}\\
&=\exp \Big\{-\frac{1}{2} \textbf{t}^T \Sigma(1) \textbf{t}\Big\},
\end{split}
\]
i.e. 
\[
f^{(1)}_{i}(\textbf{X}) \stackrel{d}{=}N_k(\textbf{0},\Sigma(1)),
\]
with $k \times k$ covariance matrix with element in the $i$-th row and $j$-th column as follows
\[
\Sigma(1)_{i,j}=\sigma_{b}^2+\sigma_{\omega}^2 \langle x^{(i)}, x^{(j)} \rangle_{\R^I}.
\]
Observe that we can also determine the marginal distributions,
\begin{align}
\label{f1k2}
f^{(1)}_{r,i}({\textbf{X}}) \sim N(0, \Sigma(1)_{r,r}),
\end{align}
where
\[
\Sigma(1)_{r,r}=\sigma_{b}^2+\sigma_{\omega}^2 \| x^{(r)} \|^2_{\R^I}.
\]
Now, for $i\geq 1$ and $l \geq 2$, by means of (\ref{caromega}), (\ref{carbias}) and (\ref{f2}) we can write
\[
\begin{split}
\varphi_{f^{(l)}_i({\textbf{X},n}) | f^{(l-1)}_{1,\dots ,n}}(\textbf{t})& = \E [e^{\ii \textbf{t}^Tf^{(l)}_i({\textbf{X},n})}| f^{(l-1)}_{1,\dots ,n}]\\
&=\E \Big[ \exp \Big\{ \ii \textbf{t}^T \Big[ \frac{1}{\sqrt{n}}\sum_{j=1}^{n}\omega_{i,j}^{(l)}( \phi \bullet f^{(l-1)}_j({\textbf{X},n}))+ b_{i}^{(1)}\textbf{1}\Big] \Big\} | f^{(l-1)}_{1,\dots ,n} \Big]\\
&=\E \Big[ \exp \Big\{ \ii \textbf{t}^T b_{i}^{(1)}\textbf{1} +\ii \textbf{t}^T \frac{1}{\sqrt{n}}\sum_{j=1}^{n}\omega_{i,j}^{(l)}( \phi \bullet f^{(l-1)}_j({\textbf{X},n})) \Big\}| f^{(l-1)}_{1,\dots ,n}\Big]\\
&=\E \Big[ \exp \Big\{ \ii (\textbf{t}^T \textbf{1}) b_{i}^{(1)} \Big\}\Big]\prod_{j=1}^{n} \E \Big[ \exp \Big\{ \ii \omega_{i,j}^{(l)} \Big( \frac{1}{\sqrt{n}}\textbf{t}^T( \phi \bullet f^{(l-1)}_j({\textbf{X},n})) \Big) \Big\}| f^{(l-1)}_{1,\dots ,n}\Big]\\
&= \exp \Big\{-\frac{1}{2} \sigma_{b}^2(\textbf{t}^T \textbf{1})^2 \Big\} \prod_{j=1}^{n} \exp \Big\{ -\frac{1}{2n} \sigma_{\omega}^2 \Big( \textbf{t}^T ( \phi \bullet f^{(l-1)}_j({\textbf{X},n}) )\Big)^2 \Big\}\\
&= \exp \Big\{-\frac{1}{2} \Big[ \sigma_{b}^2(\textbf{t}^T \textbf{1})^2+ \frac{\sigma_{\omega}^2}{n} \sum_{j=1}^{n}  \Big( \textbf{t}^T ( \phi \bullet f^{(l-1)}_j({\textbf{X},n}) )\Big)^2 \Big]\Big\}\\
&=\exp \Big\{-\frac{1}{2} \textbf{t}^T \Sigma(l,n) \textbf{t}\Big\},
\end{split}
\]
i.e. 
\[
f^{(l)}_i({\textbf{X},n}) | f^{(l-1)}_{1,\dots ,n} \stackrel{d}{=}N_k(\textbf{0},\Sigma(l,n)),
\]
with $k \times k$ covariance matrix with element in the $i$-th row and $j$-th column as follows
\[
\Sigma(l,n)_{i,j}=\sigma_{b}^2+\frac{\sigma_{\omega}^2}{n} \Big\langle  ( \phi \bullet \textbf{F}^{(l-1)}_{i}({\textbf{X},n}) ),  ( \phi \bullet \textbf{F}^{(l-1)}_{j}({\textbf{X},n}) ) \Big\rangle_{\R^n}.
\]
Observe that we can also determine the marginal distributions,
\begin{align}
\label{flk2}
f^{(l)}_{r,i}({\textbf{X}},n) | f^{(l-1)}_{1,\dots ,n} \sim N(0, \Sigma(l,n)_{r,r}),
\end{align}
where
\[
\Sigma(l,n)_{r,r}=\sigma_b^2+\frac{\sigma_{\omega}^2}{n} \| \phi \bullet \textbf{F}^{(l-1)}_{r} (\textbf{X},n)\|_{\R^n}^2.
\]
\subsection*{SM A.1: asymptotics for the $i-th$ coordinate}

First of all, from Definition \ref{def:1}, note that since $f^{(1)}_{i}(\textbf{X})$ does not depend on $n$ we consider the limit as $n \rightarrow \infty$ only for $f^{(l)}_i({\textbf{X},n})$ for all $l \geq 2$. It comes directly from Equation \eqref{f2} that, for every fixed $l$ and $n$ the sequence $\big( f^{(l)}_{i}(\textbf{X},n) \big)_{i\geq1}$ is exchangeable. In particular, let  $p^{(l)}_n$ denote the de Finetti (random) probability measure of the exchangeable sequence $\big( f^{(l)}_{i}(\textbf{X},n) \big)_{i\geq1}$. That is, by the celebrated de Finetti representation theorem, conditionally to $p^{(l)}_n$ the $f^{(l)}_{i}(\textbf{X},n)$'s are iid as $p^{(l)}_n$. Now, let consider the induction hypothesis that, $p_n^{(l-1)} \stackrel{d}{\rightarrow} q^{(l-1)}$ as $n \rightarrow +\infty$, where $q^{(l-1)}=N_k(\textbf{0},\Sigma(l-1))$. To establish the convergence in distribution we rely on Theorem 5.3 of \cite{kallenberg2006foundations} known as Levy theorem, taking into account the point-wise convergence of the characteristic functions. Therefore we can write the following expression:
\[
\begin{split}
\varphi_{f^{(l)}_i({\textbf{X},n})} (\textbf{t}) & = \E [e^{\ii \textbf{t}^Tf^{(l)}_i({\textbf{X},n})}]\\
&= \E [ \E [e^{\ii \textbf{t}^Tf^{(l)}_i({\textbf{X},n})}| f^{(l-1)}_{1,\dots ,n}] ]\\
&=\E \Big[ \exp \Big\{-\frac{1}{2} \textbf{t}^T \Sigma(l,n) \textbf{t}\Big\} \Big] \\
&=\E \Big[  \exp \Big\{-\frac{1}{2} \Big[ \sigma_{b}^2(\textbf{t}^T \textbf{1})^2+ \frac{\sigma_{\omega}^2}{n} \sum_{j=1}^{n}  \Big( \textbf{t}^T ( \phi \bullet f^{(l-1)}_j({\textbf{X},n}) )\Big)^2 \Big]\Big\} \Big]\\
&= e^{-\frac{1}{2} \sigma_{b}^2(\textbf{t}^T \textbf{1})^2  } \E\Big[ \exp \Big\{ -\frac{\sigma_{\omega}^2}{2n} \sum_{j=1}^{n}  \Big( \textbf{t}^T ( \phi \bullet f^{(l-1)}_j({\textbf{X},n}) )\Big)^2  \Big\} \Big]\\
&= e^{-\frac{1}{2}  \sigma_{b}^2(\textbf{t}^T \textbf{1})^2  } \E \Big[ \E\Big[ \exp \Big\{ -\frac{\sigma_{\omega}^2}{2n} \sum_{j=1}^{n}  \Big( \textbf{t}^T ( \phi \bullet f^{(l-1)}_j({\textbf{X},n}) )\Big)^2  \Big\} | p^{(l-1)}_n\Big] \Big]\\
&= e^{-\frac{1}{2}  \sigma_{b}^2(\textbf{t}^T \textbf{1})^2  } \E \Big[ \prod_{j=1}^{n}\E\Big[ \exp \Big\{ -\frac{\sigma_{\omega}^2}{2n} \Big( \textbf{t}^T ( \phi \bullet f^{(l-1)}_j({\textbf{X},n}) )\Big)^2  \Big\} | p^{(l-1)}_n \Big] \Big]\\
&= e^{-\frac{1}{2} \sigma_{b}^2(\textbf{t}^T \textbf{1})^2  } \E \Big[ \prod_{j=1}^{n} \int \exp \Big\{ -\frac{\sigma_{\omega}^2}{2n} \Big( \textbf{t}^T ( \phi \bullet f )\Big)^2  \Big\} p^{(l-1)}_n (\dd f) \Big]\\
&= e^{-\frac{1}{2} \sigma_{b}^2(\textbf{t}^T \textbf{1})^2  } \E \Big[ \Big( \int \exp \Big\{ -\frac{\sigma_{\omega}^2}{2n} \Big( \textbf{t}^T ( \phi \bullet f )\Big)^2  \Big\} p^{(l-1)}_n (\dd f) \Big)^n \Big].\\
\end{split}
\]
Observe that the last integral is with respect to $k$ coordinates: i.e. $\dd f =(\dd f_1, \dots ,\dd f_k)$. Denote as $\stackrel{p}{\rightarrow}$ the convergence in probability. We will prove the following lemmas:
\begin{enumerate}
	\item[L1)] for each $l \geq 2$ and $s \geq 1$, $\mathbb{P}[p^{(l-1)}_n \in Y_s]=1$, where $Y_s=\{p:\int \| \phi \bullet f \|_{\R^k}^{s} p(\dd f)<+\infty\}$;
	\item[L2)] $\int (\textbf{t}^T (\phi \bullet f))^2 p^{(l-1)}_{n} ( \dd f) \stackrel{p}{\rightarrow} \int (\textbf{t}^T (\phi \bullet f))^2  q^{(l-1)} ( \dd f)$, as $n \rightarrow +\infty$;
	\item[L3)] $\int ( \textbf{t}^T ( \phi \bullet f ))^2 \big[ 1-\exp \big\{ - \theta \frac{\sigma_{\omega}^2}{2n} ( \textbf{t}^T ( \phi \bullet f ))^2  \big\} \big] p^{(l-1)}_n (\dd f)\stackrel{p}{\rightarrow}0$, as $n \rightarrow +\infty$ for every $\theta \in (0,1)$.
\end{enumerate}

\subsubsection*{SM A.1.1: proof of L1}
In order to prove the three lemmas, we will use many times the envelope condition (\ref{envelope}) without explicit mention. For $l=2$ we have
\[
\begin{split}
\E[ \| \phi \bullet f^{(1)}_i(\textbf{X})\|_{\R^k}^{s}] & \leq \E\Big[ \Big( \sum_{r=1}^{k} |\phi \circ f^{(1)}_{r,i}(\textbf{X})|^2 \Big)^{s/2} \Big] \\
& \leq  \E \Big[ \Big( \sum_{r=1}^{k} |\phi \circ f^{(1)}_{r,i}(\textbf{X})| \Big)^{s} \Big] \\
& \leq \E \Big[ k^{s-1} \sum_{r=1}^{k} |\phi \circ f^{(1)}_{r,i}(\textbf{X})|^{s} \Big] \\
& = k^{s-1} \sum_{r=1}^k \E \Big[ |\phi \circ f^{(1)}_{r,i}(\textbf{X})|^{s} \Big] \\
& \leq k^{s-1} \sum_{r=1}^k \E \Big[ (a+b |f^{(1)}_{r,i}(\textbf{X})|^{m})^{s} \Big] \\
& \leq (2k)^{s-1} \sum_{r=1}^k \Big( a^{s} +b^{s}\E [|f^{(1)}_{r,i}(\textbf{X})|^{sm}] \Big) \\
&< +\infty,
\end{split}
\]
where we used that from (\ref{f1k2}), $f^{(1)}_{r,i}(\textbf{X}) \sim N(0,\sigma_{b}^2+\sigma_{\omega}^2 \| x^{(r)} \|^2_{\R^I})$ and then
\[
\E [|f^{(1)}_{r,i}(\textbf{X})|^{sm}]=M_{sm}(\sigma_{b}^2+\sigma_{\omega}^2 \| x^{(r)} \|^2_{\R^I})^{\nicefrac{sm}{2}},
\]
where $M_c$ is the $c$-th moment of $|N(0,1)|$. Now assume that L1 is true for $(l-2)$, i.e. for each $s \geq 1$ it holds $\int \| \phi \bullet f \|_{\R^k}^{s} p^{(l-2)}_n(\dd f)<+\infty$ uniformly in $n$, and we prove that it is true also for $(l-1)$.
\[
\begin{split}
\E[ \| \phi \bullet f_{i}^{(l-1)}(\textbf{X},n) \|_{\R^k}^s | f_{1, \dots , n}^{(l-2)} ] &\leq \E \Big[ k^{s-1} \sum_{r=1}^{k} |\phi \circ f_{r,i}^{(l-1)}(\textbf{X},n) |^s | f_{1, \dots , n}^{(l-2)} \Big]\\
& \leq (2k)^{s-1} \sum_{r=1}^{k} \Big( a^s +b^s\E \Big[  |f_{r,i}^{(l-1)}(\textbf{X},n) |^{ms} | f_{1, \dots , n}^{(l-2)} \Big] \Big)\\
& \leq D_1(a,k.s)+D_2(b,k,s) \sum_{r=1}^{k} \E \Big[  |f_{r,i}^{(l-1)}(\textbf{X},n) |^{ms} | f_{1, \dots , n}^{(l-2)} \Big]. \\
\end{split}
\]
From (\ref{flk2}) we get
\[
\begin{split}
\E \Big[  |f_{r,i}^{(l-1)}(\textbf{X},n) |^{ms} | f_{1, \dots , n}^{(l-2)} \Big] &= M_{ms} \Big( \sigma_b^2+\frac{\sigma_{\omega}^2}{n} \| \phi \bullet \textbf{F}^{(l-2)}_{r} (\textbf{X},n)\|_{\R^n}^2\Big)^{\nicefrac{sm}{2}}\\
&\leq M_{ms}2^{sm-1} \Big( \sigma_b^{2sm}+\frac{\sigma_{\omega}^{2sm}}{n^{sm}} \| \phi \bullet \textbf{F}^{(l-2)}_{r} (\textbf{X},n)\|_{\R^n}^{2sm}\Big)^{\nicefrac{1}{2}}.\\
\end{split}
\]
Thus we have
\[
\begin{split}
&\E[ \| \phi \bullet f_{i}^{(l-1)}(\textbf{X},n) \|_{\R^k}^s | p _n^{(l-2)} ] \\
& \leq D_1(a,k,s)+D_3(b,k,s,m) \sum_{r=1}^{k} \Big( \sigma_b^{2sm}+\frac{\sigma_{\omega}^{2sm}}{n^{sm}} \E\Big[ \| \phi \bullet \textbf{F}^{(l-2)}_{r} (\textbf{X},n)\|_{\R^n}^{2sm}| p^{(l-2)}_n \Big] \Big)^{\nicefrac{1}{2}},\\
\end{split}
\]
where
\[
\begin{split}
\E\Big[ \| \phi \bullet \textbf{F}^{(l-2)}_{r} (\textbf{X},n)\|_{\R^n}^{2sm}| p^{(l-2)}_n \Big]
& \leq \E \Big[ n^{sm-1} \sum_{i=1}^{n} |\phi \circ f_{r,i}^{(l-2)}(\textbf{X},n) |^{2sm} | p_n^{(l-2)} \Big]\\
& \leq D_4(s,m) n^{sm} \int |\phi (f_r)|^{2sm}  p_n^{(l-2)} (\dd f_r)\\
& \leq D_4(s,m) n^{sm} \int \|\phi \bullet f\|_{\R^k}^{2sm}  p_n^{(l-2)} (\dd f),\\
\end{split}
\]
where the last inequality is due to the fact that $ |\phi (f_r)|^{2sm} \leq \Big( \sum_{r=1}^k  | \phi(f_r) |^2 \Big)^{sm} $ and then
\[
\int |\phi (f_r)|^{2sm} p_n^{(l-2)}(\dd f_r) \leq \int \Big( \sum_{r=1}^k  | \phi(f_r) |^2 \Big)^{sm} p_n^{(l-2)}(\dd f_1 , \dots , \dd f_k)= \int \|\phi \bullet f\|_{\R^k}^{2sm}  p_n^{(l-2)} (\dd f).
\]
So, we proved that
\begin{align}
\label{Lbounded}
\begin{split}
&\E[ \| \phi \bullet f_{i}^{(l-1)}(\textbf{X},n) \|_{\R^k}^s | p _n^{(l-2)} ] \\
& \leq D_1(a,k,s)+D_3(b,k,s,m) \sum_{r=1}^{k} \Big( \sigma_b^{2sm}+\sigma_{\omega}^{2sm}D_4(s,m) \int\|\phi \bullet f\|_{\R^k}^{2sm}  p_n^{(l-2)} (\dd f) \Big)^{\nicefrac{1}{2}},\\
\end{split}
\end{align}
which is finite by induction hypothesis uniformly in $n$. To conclude, since $p_n^{(l-1)}\overset{iid}{\sim} f_i^{(l-1)}(\textbf{X},n)|p_n^{(l-1)}$ we get

\begin{align}
\label{Lbounded_2}
\begin{split}
\int \| \phi \bullet f\|_{\R^k}^s p _n^{(l-1)} (\dd f) &= \E[\|  \phi \bullet f_{i}^{(l-1)}(\textbf{X},n) \|_{\R^k}^s|p_n^{(l-1)}]\\
&=\E[ \E[\|\phi \bullet f_{i}^{(l-1)}(\textbf{X},n) \|_{\R^k}^s|p_n^{(l-2)}]|p_n^{(l-1)}]\\
&\leq cost(a,k,s,m) < \infty
\end{split}
\end{align}
which is bounded uniformly in $n$ since the inner expectation is bounded uniformly in $n$ by (\ref{Lbounded}).

\textbf{Remark}: $Y_s$ is a measurable set with respect to the weak topology for each $s\geq 1$, indeed for each $R \in \N$ defining the map
	\[ 
	T_R:U \rightarrow \R, \quad T_R(p)=\int_{B_R(0)} \| \phi \bullet f \|_{\R^k}^{s} p(\dd f) = \int_{\R^k} \| \phi \bullet f \|_{\R^k}^{s} \mathcal{X}_{(B_R(0))}(f) p(\dd f)
	\]
	where $U:=\{p: p $ distribution of a r.v. $X:\Omega \rightarrow \R^k \}$ endowed with the weak topology, since $\cap_{R \in \N}T_R^{-1}(0,\infty)=Y_s$ and $(0,\infty)$ is open, it is sufficient to prove that $T_R$ is continuous. Let $(p_m)\subset U$ such that $p_m$ converges to $p$ with respect to the weak topology, then by Definition \ref{def:3}
	\[
	|T_R(p_m)-T_R(p)|=\Big| \int \| \phi \bullet f \|_{\R^k}^{s} \mathcal{X}_{(B_R(0))}(f) p_m(\dd f) - \int \| \phi \bullet f \|_{\R^k}^{s} \mathcal{X}_{(B_R(0))}(f) p(\dd f) \Big| \rightarrow 0
	\]
	because the function $f \mapsto \|\phi \bullet f \|_{\R^k}^{s} \mathcal{X}_{(B_R(0))}(f)$ is continuous (by composition of the continuous functions $\phi$ and $\|\|^s$) and bounded by Weierstrass theorem.

\subsubsection*{SM A.1.2: proof of L2}
By induction hypothesis, $p^{(l-1)}_n$ converges weakly to a $p^{(l-1)}$ with respect to the weak topology and the limit is degenerate, in the sense that it provides a.s. the distribution $q^{(l-1)}$. Then $p^{(l-1)}_n$ converges in probability to  $p^{(l-1)}$. Then for every sub sequence $n'$ there exists a further sub sequence $n''$ such that $p^{(l-1)}_{n''}$ converges a.s. to $p^{(l-1)}$. By induction hypothesis, $p^{(l-1)}$ is absolutely continuous with respect to the Lebesgue measure. Since $\phi$ is a.s. continuous and the sequence $\big( (\textbf{t}^T (\phi \bullet f))^2 \big)_{n \geq 1}$ uniformly integrable with respect to $p^{(l-1)}_n$ (by Cauchy-Schwarz inequality and L1 $\int \big( \textbf{t}^T ( \phi \bullet f )\big)^{2s} p^{(l-1)}_n (\dd f) \leq  \|\textbf{t}\|^{2s}_{\R^k} \int \| \phi \bullet f \|^{2s}_{\R^k} p^{(l-1)}_n (\dd f) < \infty$, thus is $L^s$-bounded for each $s\geq1$, and so uniformly integrable, then we can write the following 
\[
\int (\textbf{t}^T (\phi \bullet f))^2 p^{(l-1)}_{n''} (\dd f) \stackrel{a.s.}{\rightarrow} \int (\textbf{t}^T (\phi \bullet f))^2  q^{(l-1)} (\dd f).
\]
Thus, as $n \to +\infty$
\[
\int (\textbf{t}^T (\phi \bullet f))^2 p^{(l-1)}_{n} (\dd f) \stackrel{p}{\rightarrow} \int (\textbf{t}^T (\phi \bullet f))^2  q^{(l-1)} (\dd f).
\]

\subsubsection*{SM A.1.3: proof of L3}
Let $p \geq 1$ and $q \geq 1$ such that $\frac{1}{p}+\frac{1}{q}=1$. By means of H\"older inequality
\[
\begin{split}
&\int \| \phi \bullet f \|_{\R^k}^2(1-e^{-\frac{\sigma_{\omega}^2}{2n} (\textbf{t}^T (\phi \bullet f ))^2})p^{(l-1)}_n(\dd f) \\
&\leq \Big( \int \| \phi \bullet f \|_{\R^k}^{2p} p^{(l-1)}_n(\dd f) \Big)^{\nicefrac{1}{p}} \Big( \int (1-e^{-\frac{\sigma_{\omega}^2}{2n} (\textbf{t}^T (\phi \bullet f ))^2})^q p^{(l-1)}_n(\dd f) \Big)^{\nicefrac{1}{q}}.
\end{split}
\]
Since $q\geq1$, for every $y\geq0$ we have $0\leq1-e^{-y}<1$, then $(1-e^{-y})^q \leq (1-e^{-y}) \leq y$. It implies the following
\[
\begin{split}
&\int \| \phi \bullet f \|_{\R^k}^2(1-e^{-\frac{\sigma_{\omega}^2}{2n} (\textbf{t}^T (\phi \bullet f ))^2})p^{(l-1)}_n(\dd f) \\
& \leq \Big( \int \| \phi \bullet f \|_{\R^k}^{2p} p^{(l-1)}_n(\dd f) \Big)^{\nicefrac{1}{p}} \Big( \int  \frac{\sigma_{\omega}^2}{2n} (\textbf{t}^T (\phi \bullet f ))^2p^{(l-1)}_n(\dd f) \Big)^{\nicefrac{1}{q}}\\
& \leq \Big( \int \| \phi \bullet f  \|_{\R^k}^{2p} p^{(l-1)}_n(\dd f) \Big)^{\nicefrac{1}{p}} \Big( \| \textbf{t} \|_{\R^k}^2 \frac{\sigma_{\omega}^2}{2n} \int \| \phi \bullet f \|_{\R^k}^2 p^{(l-1)}_n(\dd f) \Big)^{\nicefrac{1}{q}} \rightarrow 0,
\end{split}
\]
as $n \rightarrow +\infty$ since by L1 the two integrals are bounded uniformly in n. Thus for every $y>0$ and $\theta \in (0,1)$  $e^{-\theta y} \geq e^{-y} \Rightarrow 0\leq 1-e^{-\theta y} \leq 1-e^{- y} \leq 1$ we get
\[
\begin{split}
0 &\leq \int \big( \textbf{t}^T ( \phi \bullet f )\big)^2 \Big[ 1-\exp \big\{ -\theta \frac{\sigma_{\omega}^2}{2n} \big( \textbf{t}^T ( \phi \bullet f )\big)^2  \big\} \Big] p^{(l-1)}_n (\dd f) \\
& \leq  \| \textbf{t} \|_{\R^k}^2 \int \| \phi \bullet f \|_{\R^k}^2 \Big[ 1-\exp \big\{ - \frac{\sigma_{\omega}^2}{2n} \big( \textbf{t}^T ( \phi \bullet f )\big)^2  \big\} \Big] p^{(l-1)}_n (\dd f)\rightarrow 0,
\end{split}
\]
as $n \rightarrow +\infty$.
\subsubsection*{SM A.1.4: combination of the lemmas}
We conclude in two steps.

\textbf{Step 1: uniform integrability.}
Define $y=y_n(f)=\frac{\sigma_{\omega}^2}{2n} ( \textbf{t}^T ( \phi \bullet f ))^2 $. Thus 
\[
\begin{split}
\varphi_{f^{(l)}_i({\textbf{X},n})} (\textbf{t}) &= e^{-\frac{1}{2} \sigma_{b}^2(\textbf{t}^T \textbf{1})^2} \E \Big[ \Big( \int e^{-y_n(f)} p^{(l-1)}_n (\dd f) \Big)^n \Big] \\
&= e^{-\frac{1}{2} \sigma_{b}^2(\textbf{t}^T \textbf{1})^2  } \E [ A_n ] 
\end{split}
\]

where $A_n=\Big( \int e^{-y_n(f)} p^{(l-1)}_n (\dd f) \Big)^n $. $(A_n)_{n \geq 1}$ is is uniformly integrable because it is $L^s$-bounded for all $s \geq 1$. Indeed, since $0<e^{-y_n(f)}\leq 1$

\[
\E[A_n^s] \leq \E \Big[ \Big( \int p_n^{(l-1)}(\dd f) \Big)^{ns}\Big] = \E \big[ 1 \big] = 1
\]

\textbf{Step 2: convergence in probability.}
By Lagrange theorem for $y >0$ there exists $\theta \in (0,1)$ such that $e^{-y}=1-y+y(1-e^{-y\theta})$. Then for every $n$ there exists a real value $\theta_n \in (0,1)$ such that the follow equality holds:
\[
\begin{split}
A_n=\Big(1- \frac{\sigma_{\omega}^2}{2n} [A'_n - A''_n ]\Big)^n.
\end{split}
\]
where

\[
\begin{sistema}
A'_n= \int ( \textbf{t}^T ( \phi \bullet f ))^2 p^{(l-1)}_n (\dd f) \\
A''_n= \int ( \textbf{t}^T ( \phi \bullet f ))^2 \Big[ 1-\exp \Big\{ -\theta_n \frac{\sigma_{\omega}^2}{2n} ( \textbf{t}^T ( \phi \bullet f ))^2  \Big\} \Big] p^{(l-1)}_n (\dd f)
\end{sistema}
\]

Using the definition of the exponential function, i.e. $e^x= \lim _{n \to \infty} (1+\frac{x}{n})^n$, L2 and L3 we get that

\[
A_n \overset{p}{\rightarrow} \exp \Big\{- \frac{\sigma_{\omega}^2}{2} \int ( \textbf{t}^T ( \phi \bullet f ))^2 q^{(l-1)}(\dd f)\Big\}, \quad \text{ as } n \rightarrow \infty
\]

\textbf{Conclusion}: since convergence in probability with uniform integrability implies convergence in mean, by the two above steps we get
\[
\begin{split} 
\varphi_{f^{(l)}_i({\textbf{X},n})} (\textbf{t}) = e^{-\frac{1}{2} \sigma_{b}^2(\textbf{t}^T \textbf{1})^2  } \E [ A_n ] &\rightarrow \exp \Big\{ -\frac{\sigma_{b}^2}{2}(\textbf{t}^T \textbf{1})^2 - \frac{\sigma_{\omega}^2}{2} \int ( \textbf{t}^T ( \phi \bullet f ))^2 q^{(l-1)}(\dd f) \Big\}\\
&= \exp \Big\{ -\frac{1}{2} \Big[ \sigma_{b}^2(\textbf{t}^T \textbf{1})^2 +\sigma_{\omega}^2 \int ( \textbf{t}^T ( \phi \bullet f ))^2 q^{(l-1)}(\dd f) \Big] \Big\}\\
&= \exp \Big\{ -\frac{1}{2} \textbf{t}^T \Sigma(l) \textbf{t}\Big\},
\end{split}
\]
where $\Sigma(l)$ is a $k \times k$ matrix with elements
\[
\Sigma(l)_{i,j}= \sigma_b^2+\sigma_{\omega}^2 \int\phi( f_i) \phi (f_j) q^{(l-1)}(\dd f),
\]
where $q^{(l-1)}=N_k(\textbf{0},\Sigma(l-1))$. Then the limit distribution of $f^{(l)}_i({\textbf{X}})$ is a $k$-dimensional Gaussian distribution with mean $\textbf{0}$ and covariance matrix $\Sigma(l)$, i.e. as $n \rightarrow +\infty$,
\[
f^{(l)}_i(\textbf{X},n) \stackrel{d}{\rightarrow} N_k(\textbf{0},\Sigma(l)).
\]

\section*{SM B.1}\zlabel{sec:appB1}
Fix $i \geq 1$, $l\geq1$, $n \in \N$ . We prove that there exists a random variable $H^{(l)}_i(n)$ such that 
\[
|f^{(l)}_i(x,n)-f^{(l)}_i(y,n)|  \leq H^{(l)}_i(n)\| x-y \|_{\R^I}, \quad x,y \in \R^I,\mathbb{P}-a.s.
\]
i.e. fixed $\xi \in \Omega$ the function $x \mapsto f^{(l)}_i(x,n)(\xi)$ is Lipschitz. We proceed by induction on the layers. Fix $x,y \in \R^I$. For the first layer, by (\ref{f1}) we get
\[
\begin{split}
|f^{(1)}_i(x,n)(\xi)-f^{(1)}_i(y,n)(\xi)| &= \Big| \sum_{j=1}^{I}\omega_{i,j}^{(1)}(\xi)x_j + b_i^{(1)}(\xi) - \big( \sum_{j=1}^{I}\omega_{i,j}^{(1)}(\xi)y_j + b_i^{(1)}(\xi)\big) \Big| \\
& = \Big| \sum_{j=1}^{I}\omega_{i,j}^{(1)}(\xi)x_j -  \sum_{j=1}^{I}\omega_{i,j}^{(1)}(\xi)y_j  \Big| \\
& = \Big| \sum_{j=1}^{I}\omega_{i,j}^{(1)}(\xi) (x_j - y_j)  \Big| \\
& \leq \sum_{j=1}^{I}  \Big| \omega_{i,j}^{(1)}(\xi) \Big ||x_j - y_j| \\
&  \leq  \|x - y\|_{\R^I} \sum_{j=1}^{I}  \Big| \omega_{i,j}^{(1)}(\xi) \Big | \\
\end{split}
\]

where we used that $|x_j-y_j|\leq \| x-y \|_{\R^I}$. Set $H^{(1)}_i(n)= \sum_{j=1}^{I}\big|\omega_{i,j}^{(1)} \big |$. Suppose by induction hypothesis that for each $j \geq 1$ there exists a random variable $H_j^{(l-1)}(n)$ such that $|f^{(l-1)}_j(x,n)(\xi)-f^{(l-1)}_j(y,n)(\xi)| \leq  H_j^{(l-1)}(n)(\xi) \|x - y\|_{\R^I}$, and let $L_{\phi}$ be the Lipschitz constant of $\phi$. Then by (\ref{f2}) we get
\[
\begin{split}
|f^{(l)}_i(x,n)&(\xi)-f^{(l)}_i(y,n)(\xi)| \\
&= \Big| \frac{1}{\sqrt{n}}\sum_{j=1}^{n}\omega_{i,j}^{(l)}(\xi)\phi(f^{(l-1)}_j(x,n)) + b_i^{(l)}(\xi) - \Big[ \frac{1}{\sqrt{n}} \sum_{j=1}^{n}\omega_{i,j}^{(l)}(\xi)\phi(f^{(l-1)}_j(y,n)) + b_i^{(l)}(\xi)\Big] \Big| \\
&= \Big| \frac{1}{\sqrt{n}}\sum_{j=1}^{n}\omega_{i,j}^{(l)}(\xi)\phi(f^{(l-1)}_j(x,n)) - \frac{1}{\sqrt{n}} \sum_{j=1}^{n}\omega_{i,j}^{(l)}(\xi)\phi(f^{(l-1)}_j(y,n)) \Big| \\
& \leq \frac{1}{\sqrt{n}} \sum_{j=1}^{n} \Big| \omega_{i,j}^{(l)}(\xi) \Big| \big| \phi(f^{(l-1)}_j(x,n)) - \phi(f^{(l-1)}_j(y,n)) \big| \\
& \leq \frac{L_{\phi}}{\sqrt{n}} \sum_{j=1}^{n} \Big| \omega_{i,j}^{(l)}(\xi)\Big| \big| f^{(l-1)}_j(x,n) - f^{(l-1)}_j(y,n) \big| \\
& \leq \frac{L_{\phi}}{\sqrt{n}} \sum_{j=1}^{n} \Big| \omega_{i,j}^{(l)}(\xi)\Big|  H^{(l-1)}_j(n)(\xi) \|x - y\|_{\R^I} \\
& \leq \| x-y \|_{\R^I} \frac{L_{\phi}}{\sqrt{n}} \sum_{j=1}^{n} \Big| \omega_{i,j}^{(l)}(\xi)\Big|  H^{(l-1)}_j(n)(\xi) \\
\end{split}
\]
Set 
\begin{align*}\label{Lipsch_i_n}
H^{(l)}_i(n)=\frac{L_{\phi}}{\sqrt{n}} \sum_{j=1}^{n} \Big| \omega_{i,j}^{(l)}\Big|  H^{(l-1)}_j(n)
\end{align*}
Thus we proved that fixed $l \geq 1$, and $i \geq 1$, for each $n\in \N$
\[
\mathbb{P}\Big[ \Big\{ \xi \in \Omega: |f^{(l)}_i(x,n)(\xi) -f^{(l)}_i(y,n)(\xi)|  \leq H^{(l)}_i(n)(\xi)\| x-y \|_{\R^I} \Big\} \Big]=1.
\]
Thus, each process $ f^{(l)}_i(1), f^{(l)}_i(2), \dots$ is $\mathbb{P}$-a.s. Lipschitz, in particular is $\mathbb{P}$-a.s. continuous processes, i.e. belongs to $C(\R^I;\R)$. In order to prove the continuity of $f^{(l)}_i$ we can not just take the limit as $n \rightarrow +\infty$ of (\ref{Lipschitz}) because the left quantity converges to $|f_i^{(l)}(x)-f^{(l)}_i(y)|$ only in distribution and not $\mathbb{P}$-a.s., but we can prove the continuity by applying Proposition \ref{Kolmogorov_Chentsov_0}, as we will show in \ztitleref{sec:appB2}.

\section*{SM B.2}\zlabel{sec:appB2}
Fix $i \geq 1$, $l\geq1$. We show the continuity of the limiting process $f^{(l)}_i$ by applying Proposition \ref{Kolmogorov_Chentsov_0}. Take two inputs $x,y \in \R^I$. From (\ref{thm2}) we know that $[f^{(l)}_i(x),f^{(l)}_i(y)] \sim N_2(\textbf{0}, \Sigma(l))$ where
\[
 \begin{split}
& \Sigma(1)=
 \sigma_b^2
\begin{bmatrix}
 1 & 1 \\
 1 & 1 \\
 \end{bmatrix}
 +  \sigma_{\omega}^2 
\begin{bmatrix} 
\| x \|_{\R^I}^2 & \langle x , y \rangle_{\R^I} \\
\langle x , y \rangle_{\R^I} & \| y \|_{\R^I}^2 \\
\end{bmatrix},\\
 &\Sigma(l)= \sigma_b^2
\begin{bmatrix}
 1 & 1 \\
 1 & 1 \\
 \end{bmatrix}
 +  \sigma_{\omega}^2  \int
\begin{bmatrix} 
| \phi(u) |^2 &  \phi(u) \phi(v) \\
\phi(u) \phi(v) & | \phi(v) |^2 \\
\end{bmatrix}
q^{(l-1)} (\dd u, \dd v),
\end{split}
\]
where $q^{(l-1)}= N_2(\textbf{0}, \Sigma (l-1))$. We want to find two values $\alpha >0$ and $\beta >0$, and a constant $H^{(l)}>0$ such that
\[
\E \Big[| f^{(l)}_i(y)-f^{(l)}_i(x) |^{\alpha} \Big] \leq H^{(l)} \|y-x \|_{\R^I}^{I+\beta}.
\]
Defining $\textbf{a}^T=[1,-1]$ we have $f^{(l)}_i(y)-f^{(l)}_i(x) \sim N(\textbf{a}^T\textbf{0},\textbf{a}^T \Sigma(l) \textbf{a})$.
Consider $\alpha =2\theta$ with $\theta$ integer. Thus
\[
|f^{(l)}_i(y)-f^{(l)}_i(x)|^{2\theta} \sim | \sqrt{\textbf{a}^T\Sigma(l) \textbf{a}} N(0,1) |^{2\theta} \sim (\textbf{a}^T\Sigma(l) \textbf{a})^{\theta}  |N(0,1)|^{2\theta}.
\]
We proceed by induction over the layers. For $l=1$,
\[
\begin{split}
\E \Big[| f^{(1)}_i(y)-f^{(1)}_i(x) |^{2\theta} \Big] &= C_{\theta} (\textbf{a}^T\Sigma(1) \textbf{a})^{\theta}\\
&=C_{\theta} (\sigma_{\omega}^2 \|y \|_{\R^I}^{2} -2\sigma_{\omega}^2 \langle y,x \rangle_{\R^I} +\sigma_{\omega}^2 \| x \|_{\R^I}^2 )^{\theta}\\
&=C_{\theta} (\sigma_{\omega}^2)^{\theta} ( \|y \|_{\R^I}^{2} -2\langle y,x \rangle_{\R^I} + \| x \|_{\R^I}^2 )^{\theta}\\
&= C_{\theta} (\sigma_{\omega}^2)^{\theta} \|y-x \|_{\R^I}^{2\theta},
\end{split}
\]
where $C_{\theta}$ is the $\theta$-th moment of the chi-square distribution with one degree of freedom. By hypothesis $\phi$ is Lipschitz.
\[
\int |u-v|^{2\theta} q^{(l-1)}(\dd u, \dd v) \leq H^{(l-1)} \|y-x \|^{2\theta}_{\R^I}.
\]
Then,
\[
\begin{split}
| f^{(l)}_i(y)-f^{(l)}_i(x) |^{2\theta} &\sim |N(0,1)|^{2\theta} (\textbf{a}^T\Sigma(l) \textbf{a})^{\theta}\\
& = |N(0,1)|^{2\theta} \Big(  \sigma^2_{\omega}\int  [ |\phi(u)|^2-2\phi(u)\phi(v)+|\phi(v)|^2 ] q^{(l-1)}(\dd u, \dd v) \Big)^{\theta}\\
& = |N(0,1)|^{2\theta} \Big(  \sigma^2_{\omega}\int | \phi(u)-\phi(v) |^2 q^{(l-1)}(\dd u, \dd v) \Big)^{\theta}\\
& \leq |N(0,1)|^{2\theta} (\sigma^2_{\omega} L_{\phi}^2)^{\theta} \Big(  \int |u-v|^2 q^{(l-1)}(\dd u, \dd v) \Big)^{\theta}\\
& \leq |N(0,1)|^{2\theta} (\sigma^2_{\omega} L_{\phi}^2)^{\theta} \int |u-v|^{2\theta} q^{(l-1)}(\dd u, \dd v)\\
& \leq |N(0,1)|^{2\theta} (\sigma^2_{\omega} L_{\phi}^2)^{\theta} H^{(l-1)}\| y-x \|_{\R^{I}}^{2\theta}.
\end{split}
\]
Thus we conclude
\[
\begin{split}
\E \Big[| f^{(l)}_i(y)-f^{(l)}_i(x) |^{2\theta} \Big] \leq H^{(l)}\|y - x \|_{\R^I}^{2\theta},
\end{split}
\]
where the constant $H^{(l)}$ can be explicitly derived by solving the following system
\[
\begin{sistema}
H^{(1)} = C_{\theta} (\sigma^2_{\omega})^{\theta} \\
H^{(l)} = C_{\theta} (\sigma^2_{\omega} L_{\phi}^2)^{\theta} H^{(l-1)}.
\end{sistema}
\]
It is easy to get $H^{(l)}= C_{\theta}^l(\sigma^2_{\omega})^{l\theta} (L_{\phi}^2)^{(l-1)\theta}$. Notice that this quantity does not depend on $i$. Therefore, by Proposition \ref{Kolmogorov_Chentsov_0}, by placing $\alpha=2\theta$ and $\beta=2\theta-I$, for every $\theta>I/2$ ($\beta$ needs to be positive then we take $\theta>I/2$) there exists a continuous version $f^{(l)(\theta)}_i$ of the process $f^{(l)}_i$ with $\mathbb{P}$-a.s. locally $\gamma$-H\"older paths for every $0<\gamma<1-\frac{I}{2\theta}$.
\begin{itemize}
\item Thus $f^{(l)(\theta)}_i$ and $f^{(l)}_i$ are indistinguishable (same trajectories), i.e there exists $\Omega^{(\theta)}\subset \Omega$ with $\mathbb{P} (\Omega^{(\theta)})=1$ such that for each $\omega \in \Omega^{(\theta)}$, $x \mapsto f^{(l)}_i(x)(\omega)$ is locally $\gamma$-H\"older for each $0<\gamma<1-\frac{I}{2\theta}$.
\item Define $\Omega^{\star}= \bigcap_{\theta > I/2}\Omega^{(\theta)}$, then for each $0 < \delta_0 < 1$ there exists $\theta_0$ such that $\delta_0 < 1- \frac{I}{2\theta_0} < 1$, thus for each $\omega \in \Omega^{\star} \subset \Omega^{(\theta_0)}$, the trajectory $x \mapsto f^{(l)}_i(x)(\omega)$ is locally $\delta_0$-H\"older continuous.
\end{itemize}
By Proposition \ref{Kolmogorov_Chentsov_0} we can conclude that $f^{(l)}_i$ has a continuous version and the latter is $\mathbb{P}$-a.s locally $\gamma$-H\"older continuous for every $0<\gamma<1$.

\section*{SM B.3}\zlabel{sec:appB3}
Fix $i \geq 1$, $l\geq1$. We apply Proposition \ref{Kolmogorov_Chentsov} to show the uniform tightness of the sequence $(f^{(l)}_i(n))_{n\geq 1}$ in $C(\R^I;\R)$. By Lemma \ref{Lemma2} $f^{(l)}_i(1),f^{(l)}_i(2), \dots$ are random elements in $C(\R^I;\R)$.  First we show that the sequence $f(0_{\R^I},n)_{n \geq 1}$ is uniformly tight in $\R$. We use the following statement from \cite[Theorem 11.5.3]{dudley2002real}
\begin{prop} \label{Unif_tight}
Let $(C,\rho)$ be a metric space and suppose $f(n)\stackrel{d}{\rightarrow}f$ where $f(n)$ is tight for all $n$. Then $f(n)_{n \geq 1}$ is uniformly tight. 
\end{prop}
Since $(\R,|\cdot|)$ is Polish every probability measure is tight, then $f(0_{\R^I},n)$ is tight in $\R$ for every $n$. Moreover, by Lemma \ref{Lemma1} $f_i(0_{\R^I},n)_{n \geq 1}\stackrel{d}{\rightarrow}f_i^{(l)}(0_{\R^I})$, then by Proposition (\ref{Unif_tight}) $f(0_{\R^I},n)_{n \geq 1}$ is uniformly tight in $\R$. In order to apply Proposition \ref{Kolmogorov_Chentsov} it remains to show that there exist two values $\alpha >0$ and $\beta >0$, and a constant $H^{(l)}>0$ such that
\[
\E \Big[| f_{i}^{(l)}(y,n)-f_{i}^{(l)}(x,n) |^{\alpha} \Big] \leq H^{(l)} \|y-x \|_{\R^I}^{I+\beta}, \quad x,y \in \R^I, n \in \N
\]
uniformly in $n$. The first idea could be try to bound (uniformly in $n$) the expected value of $H^{(l)}_i(n)$ obtained in (\ref{Lipsch_i_n}), but this turns out to be very difficult.  Thus we choose another way. Take two points $x,y \in \R^I$. From (\ref{f1k}) we know that $f_{i}^{(l)}(y,n)|f^{(l-1)}_{1,\dots ,n} \sim N(0,\sigma^2_y(l,n))$ and $f_{i}^{(l)}(x,n)|f^{(l-1)}_{1,\dots ,n} \sim N(0, \sigma^2_x(l,n))$ with joint distribution $N_2(\textbf{0}, \Sigma(l,n))$, where

\[
  \begin{split}
& \Sigma(1)=
\begin{bmatrix}
  \sigma^2_x(1) & \Sigma(1)_{x,y} \\
  \Sigma(1)_{x,y} & \sigma^2_y(1) \\
  \end{bmatrix},
  \\
  &\Sigma(l)= 
\begin{bmatrix}
  \sigma^2_x(l,n) & \Sigma(l,n)_{x,y} \\
  \Sigma(l,n)_{x,y} & \sigma^2_y(l,n) \\
  \end{bmatrix},
\end{split}
\]
with,
\[
\begin{split}
& \sigma^2_x(1) = \sigma_{b}^2 + \sigma_{\omega}^2 \| x \|_{\R^I}^2,\\
& \sigma^2_y(1) = \sigma_{b}^2 + \sigma_{\omega}^2 \| y \|_{\R^I}^2,\\
& \Sigma(1)_{x,y} = \sigma_{b}^2 + \sigma_{\omega}^2 \langle x , y \rangle_{\R^I},\\
&  \sigma^2_x(l,n) = \sigma_{b}^2 + \frac{\sigma_{\omega}^2}{n} \sum_{j=1}^n | \phi \circ f_j^{(l-1)} (x,n) |^2, \\
& \sigma^2_y(l,n) = \sigma_{b}^2 + \frac{\sigma_{\omega}^2}{n} \sum_{j=1}^n | \phi \circ f_j^{(l-1)} (y,n) |^2, \\
& \Sigma(l,n)_{x,y} = \sigma_{b}^2 + \frac{\sigma_{\omega}^2}{n} \sum_{j=1}^n \phi( f^{(l-1)}_j(x,n)) \phi(f^{(l-1)}_j(y,n))
\end{split}
\]
Defining $\textbf{a}^T=[1,-1]$ we have that $f_{i}^{(l)}(y,n)|f^{(l-1)}_{1,\dots ,n}-f_{i}^{(l)}(x,n)|f^{(l-1)}_{1,\dots ,n}$ is distributed as $N(\textbf{a}^T\textbf{0},\textbf{a}^T \Sigma(l,n) \textbf{a})$, where
\[
\textbf{a}^T \Sigma(l,n) \textbf{a} = \sigma_y^2(l,n)-2\Sigma(l,n)_{x,y}+\sigma_x^2(l,n).
\] 
Consider $\alpha=2\theta$ with $\theta$ integer. Thus
\[
\Big|f_{i}^{(l)}(y,n)|f^{(l-1)}_{1,\dots ,n}-f_{i}^{(l)}(x,n)|f^{(l-1)}_{1,\dots ,n}\Big|^{2\theta} \sim | \sqrt{\textbf{a}^T\Sigma(l,n) \textbf{a}} N(0,1) |^{2\theta} \sim (\textbf{a}^T\Sigma(l,n) \textbf{a})^{\theta}  |N(0,1)| ^{2\theta}.
\]
Start first with the case $l=1$.
\[
\begin{split}
\E \Big[| f_i^{(1)}(y,n)-f_i^{(1)}(x,n)|^{2\theta} \Big] &= C_{\theta} (\textbf{a}^T\Sigma(1) \textbf{a})^{\theta}\\
&=C_{\theta} (\sigma_{\omega}^2 \|y \|_{\R^I}^{2} -2\sigma_{\omega}^2 \langle y,x \rangle_{\R^I} +\sigma_{\omega}^2 \| x \|_{\R^I}^2 )^{\theta}\\
&=C_{\theta} (\sigma_{\omega}^2)^{\theta} ( \|y \|_{\R^I}^{2} -2\langle y,x \rangle_{\R^I} + \| x \|_{\R^I}^2 )^{\theta}\\
&= C_{\theta} (\sigma_{\omega}^2)^{\theta} \|y-x \|_{\R^I}^{2\theta},
\end{split}
\]
where $C_{\theta}$ is the $\theta$-th moment of the chi-square distribution with one degree of freedom. Set $H^{(1)}=C_{\theta} (\sigma_{\omega}^2)^{\theta}$.  By hypothesis induction suppose that for every $j\geq 1$
\[
\E \Big[| f^{(l-1)}_j (y,n)- f^{(l-1)}_j (x,n) |^{2\theta}\Big] \leq H^{(l-1)} \| y-x \|_{\R^I}^{2\theta}.
\]
By hypothesis $\phi$ is Lipschitz, then
\[
\begin{split}
\E\Big[ | f_{i}^{(l)}(y,n)-f_{i}^{(l)}(x,n) |^{2\theta} \Big|f^{(l-1)}_{1,\dots ,n} \Big] & = C_{\theta} (\textbf{a}^T\Sigma(l,n) \textbf{a})^{\theta}\\
&=C_{\theta}  \Big( \sigma_y^2(l,n)-2\Sigma(l,n)_{x,y}+\sigma_x^2(l,n)  \Big)^{\theta}\\
&= C_{\theta}  \Big( \frac{\sigma^2_{\omega}}{n} \sum_{j=1}^n \Big|  \phi \circ f_j^{(l-1)} (y,n) - \phi \circ f_j^{(l-1)} (x,n) \Big|^2  \Big)^{\theta} \\
&\leq C_{\theta}  \Big( \frac{\sigma^2_{\omega}L_{\phi}^2}{n} \sum_{j=1}^n\Big|  f_j^{(l-1)} (y,n)  - f_j^{(l-1)} (x,n)  \Big|^2 \Big)^{\theta}\\
&= C_{\theta}  \frac{(\sigma^2_{\omega}L_{\phi}^2)^{\theta}}{n^{\theta}} \Big( \sum_{j=1}^n\Big|  f_j^{(l-1)} (y,n)  - f_j^{(l-1)} (x,n)  \Big|^2 \Big)^{\theta}\\
&\leq C_{\theta}  \frac{(\sigma^2_{\omega}L_{\phi}^2)^{\theta}}{n} \sum_{j=1}^n\Big|  f_j^{(l-1)} (y,n)  - f_j^{(l-1)} (x,n)  \Big|^{2\theta}.\\
\end{split}
\]
Using the induction hypothesis 
\[
\begin{split}
\E \Big[ | f_{i}^{(l)}(y,n)-f_{i}^{(l)}(x,n) |^{2\theta} \Big] &=\E \Big[ \E \Big[ | f_{i}^{(l)}(y,n)-f_{i}^{(l)}(x,n) |^{2\theta} \Big|f^{(l-1)}_{1,\dots ,n} \Big]\Big] \\
& \leq C_{\theta} \frac{(\sigma^2_{\omega} L_{\phi}^2)^{\theta}}{n} \sum_{j=1}^n \E \Big[ |  f_j^{(l-1)} (y,n)  - f_j^{(l-1)} (x,n) |^{2\theta} \Big] \\
& \leq C_{\theta} (\sigma^2_{\omega} L_{\phi}^2)^{\theta} H^{(l-1)} \| y-x \|^{2\theta}_{\R^I}.
\end{split}
\]
We can get the constant $H^{(l)}$ by solving the same system as (\ref{H^l}), obtaining $H^{(l)}= C_{\theta}^l(\sigma^2_{\omega})^{l\theta} (L_{\phi}^2)^{(l-1)\theta}$ which does not depend on $n$. By Proposition \ref{Kolmogorov_Chentsov} setting $\alpha=2\theta$ and $\beta= 2\theta-I$, since $\beta$ must be a positive constant, it is sufficient to take $\theta >I/2$ and this concludes the proof.

\section*{SM C}\zlabel{sec:appC}
Fix $k$ inputs $\textbf{X}=[x^{(1)}, \dots , x^{(k)}]$ and a layer $l$. We show that as $n \rightarrow +\infty$
\begin{align*}
\Big( f^{(l)}_i(\textbf{X},n) \Big)_{i \geq 1} \stackrel{d}{\rightarrow} \bigotimes_{i=1}^{\infty} N_k(\textbf{0},\Sigma(l))
\end{align*}
where $\bigotimes$ denotes the product measure and with $\Sigma(l)$ as in (\ref{thm2}).
We prove this statement by proving the $n$ large asymptotic behaviour of any finite linear combination of the $f^{(l)}_i(\textbf{X},n)$'s, for $i \in \mathcal L \subset \N$. See, e.g. \cite{billingsley1999convergence} for details. Following the notation of \cite{matthews2018gaussianB}, consider a finite linear combination of the function values without the bias, i.e.,
\[
\mathcal T^{(l)}(\mathcal L,p,\textbf{X},n)= \sum_{i \in \mathcal L} p_i[f^{(l)}_i(\textbf{X},n)-b_i^{(l)}\textbf{1}].
\]
Then for the first layer we write 
\[
\begin{split}
\mathcal T^{(1)}(\mathcal L,p,\textbf{X}) & = \sum_{i \in \mathcal L} p_i \Big[ \sum_{j=1}^I \omega_{i,j}^{(1)} \textbf{x}_{j} \Big]\\
&= \sum_{j=1}^I \gamma_j^{(1)}(\mathcal L, p, \textbf{X}),
\end{split}
\]
where 

\[
\begin{split}
\gamma_j^{(1)}(\mathcal L, p, \textbf{X})=\sum_{i \in \mathcal L} p_i \omega_{i,j}^{(l)}\textbf{x}_j.
\end{split}
\]

and for any $l \geq 2$
\[
\begin{split}
\mathcal T^{(l)}(\mathcal L,p,\textbf{X},n) & = \sum_{i \in \mathcal L} p_i \Big[ \frac{1}{\sqrt{n}}\sum_{j=1}^n \omega_{i,j}^{(l)}(\phi \bullet f_j^{(l-1)}(\textbf{X},n)) \Big]\\
&=\frac{1}{\sqrt{n}} \sum_{j=1}^n \gamma_j^{(l)}(\mathcal L, p, \textbf{X},n),
\end{split}
\]
where
\[
\begin{split}
\gamma_j^{(l)}(\mathcal L, p, \textbf{X},n)=\sum_{i \in \mathcal L} p_i \omega_{i,j}^{(l)}(\phi \bullet f_j^{(l-1)}(\textbf{X},n)).
\end{split}
\]
For the first layer we get

\[
\begin{split}
\varphi_{\mathcal T^{(1)}(\mathcal L,p,\textbf{X})}(\textbf{t})&=\E \Big[ e^{\ii \textbf{t}^T \mathcal T^{(1)}(\mathcal L,p,\textbf{X})} \Big] \\
&= \E \Big[ \exp \Big\{ \ii \textbf{t}^T \Big[ \sum_{j=1}^I \sum_{i \in \mathcal L} p_i \omega_{i,j}^{(1)}\textbf{x}_j \Big] \Big\} \Big] \\
&= \prod_{j=1}^I \prod_{i \in \mathcal L} \E \Big[ \exp \Big\{ \ii \textbf{t}^T \Big[ p_i \omega_{i,j}^{(1)} \textbf{x}_j \Big] \Big\} \Big] \\
&= \prod_{j=1}^I \prod_{i \in \mathcal L} \exp \Big\{ -\frac{\sigma^2_{\omega}}{2} p_i^2 \Big( \textbf{t}^T \textbf{x}_j \Big)^2 \Big\}  \\
&= \exp \Big\{ -\frac{\sigma^2_{\omega}}{2n} \sum_{i \in \mathcal L} p_i^2 \sum_{j=1}^n \Big( \textbf{t}^T \textbf{x}_j  \Big)^2 \Big\}  \\
&= \exp \Big\{ -\frac{1}{2} \textbf{t}^T \Theta(\mathcal L , p,1)\textbf{t} \Big\}, \\
\end{split} 
\]

i.e. 
\[
\mathcal T^{(1)}(\mathcal L,p,\textbf{X}) \stackrel{d}{=}N_k(\textbf{0},\Theta(\mathcal L , p,1)),
\]
with $k \times k$ covariance matrix with element in the $i$-th row and $j$-th column as follows
\[
\begin{split}
&\Theta_{i,j}(\mathcal L , p,1)=p^Tp\sigma_{\omega}^2 \langle  x^{(i)},  x^{(j)}\rangle_{\R^I},
\end{split}
\]
where $p^T p = \sum_{i \in \mathcal L} p_i^2$. For $l \geq 2$ we get
\[
\begin{split}
\varphi_{\mathcal T^{(l)}(\mathcal L,p,\textbf{X},n) | f_{1, \dots , n}^{(l-1)}}(\textbf{t})&=\E \Big[ e^{\ii \textbf{t}^T \mathcal T^{(l)}(\mathcal L,p,\textbf{X},n)} | f_{1, \dots , n}^{(l-1)} \Big] \\
&= \E \Big[ \exp \Big\{ \ii \textbf{t}^T \Big[ \frac{1}{\sqrt{n}}\sum_{j=1}^n \sum_{i \in \mathcal L} p_i \omega_{i,j}^{(l)}(\phi \bullet f_j^{(l-1)}(\textbf{X},n)) \Big] \Big\} | f_{1, \dots , n}^{(l-1)} \Big] \\
&= \prod_{j=1}^n \prod_{i \in \mathcal L} \E \Big[ \exp \Big\{ \ii \textbf{t}^T \Big[ \frac{1}{\sqrt{n}} p_i \omega_{i,j}^{(l)}(\phi \bullet f_j^{(l-1)}(\textbf{X},n)) \Big] \Big\} | f_{1, \dots , n}^{(l-1)} \Big] \\
&= \prod_{j=1}^n \prod_{i \in \mathcal L} \ \exp \Big\{ -\frac{\sigma^2_{\omega}}{2n} p_i^2 \Big( \textbf{t}^T (\phi \bullet f_j^{(l-1)}(\textbf{X},n)) \Big)^2 \Big\}  \\
&= \ \exp \Big\{ -\frac{\sigma^2_{\omega}}{2n} \sum_{i \in \mathcal L} p_i^2 \sum_{j=1}^n \Big( \textbf{t}^T (\phi \bullet f_j^{(l-1)}(\textbf{X},n)) \Big)^2 \Big\}  \\
&= \exp \Big\{ -\frac{1}{2} \textbf{t}^T \Theta(\mathcal L , p,l,n) \textbf{t} \Big\}, \\
\end{split} 
\]
i.e. 
\[
\mathcal T^{(l)}(\mathcal L,p,\textbf{X},n) | f_{1, \dots , n}^{(l-1)} \stackrel{d}{=}N_k(\textbf{0},\Theta(\mathcal L , p,l,n)),
\]
with $k \times k$ covariance matrix with element in the $i$-th row and $j$-th column as follows
\[
\begin{split}
&\Theta_{i,j}(\mathcal L , p,l,n)=p^Tp\frac{\sigma_{\omega}^2}{n} \Big\langle  ( \phi \bullet \textbf{F}^{(l-1)}_{i}({\textbf{X},n}) ),  ( \phi \bullet \textbf{F}^{(l-1)}_{j}({\textbf{X},n}) ) \Big\rangle_{\R^n},
\end{split}
\]
where $p^T p = \sum_{i \in \mathcal L} p_i^2$. Thus, along lines similar to the proof of the large $n$ asymptotics for the $i-th$ coordinate (just replacing $\sigma_b^2 \leftarrow 0$ and $\sigma^2_{\omega} \leftarrow p^Tp\sigma^2_{\omega}$), we have that for any $l \geq 2$, as $n \rightarrow +\infty$,
\[
\begin{split}
\varphi_{\mathcal T^{(l)}(\mathcal L,p,\textbf{X},n)}(\textbf{t}) & \rightarrow 
 \exp \Big\{ -\frac{1}{2} p^Tp\sigma_{\omega}^2 \int \Big( \textbf{t}^T ( \phi \bullet f )\Big)^2 q^{(l-1)}(\dd f) \Big] \Big\}\\
&= \exp \Big\{ -\frac{1}{2} \textbf{t}^T \Theta(\mathcal L , p,l) \textbf{t} \Big\}, \\
\end{split} 
\]
i.e. $\mathcal T^{(l)}(\mathcal L,p,\textbf{X},n)$ converges weakly to a $k$-dimensional Gaussian distribution with mean $\textbf{0}$ and $k \times k$ covariance matrix $\Theta(\mathcal L , p,l)$ with elements
\[
\begin{split}
&\Theta_{i,j}(\mathcal L , p,l)= p^Tp\sigma_{\omega}^2 \int\phi( f_i) \phi (f_j) q^{(l-1)}(\dd f ),
\end{split}
\]
where $q^{(l-1)}(\dd f)=q^{(l-1)}(\dd f_1, \dots, \dd f_k)=N_k(\textbf{0},\Theta(\mathcal L , p,l-1)) \dd f$. To complete the proof just observe that $\Theta(\mathcal L , p,l)=p^Tp \Sigma(l)$.

\section*{SM D.1}\zlabel{sec:appD1}

We will use, without explicit mention, that the series $\sum_{i=1}^{\infty}q^i$ converges when $|q|<1$. In particular when $q=\nicefrac{1}{2}$ the series sum to 1. Fix, $l\geq1$ and $n \in \N$ . We prove that there exists a random variable $H^{(l)}(n)$ such that 
\[
d\big(\textbf{F}^{(l)}(x,n),\textbf{F}^{(l)}(y,n)\big)_{\infty} \leq H^{(l)}(n)\| x-y \|_{\R^I}, \quad x,y \in \R^I,\mathbb{P}-a.s.
\]
It immediately derives from the Lipschitzianity of each component, indeed by (\ref{Lipschitz}) we get
\[
\begin{split}
d\big(\textbf{F}^{(l)}(x,n),\textbf{F}^{(l)}(y,n)\big)_{\infty} &= \sum_{i=1}^{\infty}\frac{1}{2^i}\frac{|f^{(l)}_i(x,n)-f^{(l)}_i(y,n)|}{1+|f^{(l)}_i(x,n)-f^{(l)}_i(y,n)|} \\
& \leq \sum_{i=1}^{\infty}\frac{1}{2^i}|f^{(l)}_i(x,n)-f^{(l)}_i(y,n)| \\
& \leq \| x-y \|_{\R^I} \sum_{i=1}^{\infty}\frac{1}{2^i}H^{(l)}_i(n).
\end{split}
\] 
It remains to show that the series $\sum_{i=1}^{\infty}\frac{1}{2^i}H^{(l)}_i(n)$ converges almost surely. By (\ref{Lipsch_i_n}) we get

\[
\begin{split}
\sum_{i=1}^{\infty}\frac{1}{2^i}H^{(l)}_i(n) &= \sum_{i=1}^{\infty}\frac{1}{2^i}\frac{L_{\phi}}{\sqrt{n}}\sum_{j=1}^n \big| \omega_{i,j}^{(l)} \big|H_j^{(l-1)}(n) \\
&= \frac{L_{\phi}}{\sqrt{n}}\sum_{j=1}^{n}H_j^{(l-1)}(n) \sum_{i=1}^{\infty}\frac{|\omega_{i,j}^{(l)}|}{2^i}.
\end{split}
\]

It remains to show the convergence almost surely of the series $\sum_{i=1}^{\infty}\frac{|\omega_{i,j}^{(l)}|}{2^i}$. We apply the three-series Kolmogorov criterion \citep[Theorem 4.18]{kallenberg2006foundations}. Call $X_i := \frac{|\omega_{i,j}^{(l)}|}{2^i}$
\begin{itemize}
\item By Markov inequality $\mathbb{P}(X_i>1)\leq \E[X_i]=\frac{\E[|N(0,\sigma_{\omega}^2)|]}{2^i}$, thus $\sum_{i=1}^{\infty}\mathbb{P}(X_i>1)\leq \E[|N(0,\sigma_{\omega}^2)|] < \infty$
\item Call $Y_i=X_i \mathbb{I}_{ \{ X_i \leq 1 \} } \leq X_i$. Then $\sum_{i=1}^{\infty} \E [Y_i] \leq \sum_{i=1}^{\infty} \E [X_i] = \sum_{i=1}^{\infty} \frac{\E [|N(0,\sigma_{\omega})|]}{2^i} = \E[|N(0,\sigma_{\omega}^2)|] < \infty$
\item $\mathrm{V}(Y_i)=\E[Y_i^2]- \E^2[Y_i]$, thus $\sum_{i=1}^{\infty}\mathrm{V}(Y_i)=\sum_{i=1}^{\infty}\E[Y_i^2]- \sum_{i=1}^{\infty}\E^2[Y_i]$. The first series converges since $\E[Y_i^2] \leq \E[X_i^2]=\frac{\sigma_{\omega}^2 \E[\chi^2(1)]}{4^i}=\frac{\sigma_{\omega}^2}{4^i}$ (then $\sum \E Y_i \leq \sigma_{\omega}^2 \sum \frac{1}{4^i}< \infty$), and the other series converges since $0 < \E[Y_i] \leq \E[X_i]$ implies $\E^2[Y_i] \leq \E^2[X_i]=\frac{\E^2[|N(0,\sigma_{\omega}^2)|]}{4^i}$ (then $\sum \E^2[Y_i] \leq \E^2[|N(0,\sigma_{\omega}^2)|] \sum \frac{1}{4^i} < \infty$).
\end{itemize}

Denoting $Q_j^{(l)} = \sum_{i=1}^{\infty}\frac{|\omega_{i,j}^{(l)}|}{2^i}$ and by setting $H^{(l)}(n):= \frac{L_{\phi}}{\sqrt{n}}\sum_{j=1}^{n}H_j^{(l-1)}(n)Q_j^{(l)}$ we complete the proof.

\section*{SM D.2}\zlabel{sec:appD2}
Fix $l\geq1$. We show the continuity of the limiting process $\textbf{F}^{(l)}$ by applying Proposition \ref{Kolmogorov_Chentsov_0}. We will use, without explicit mention, that the function $r \mapsto \frac{r}{1+r}$ is bounded by 1 for $r>0$. Take two inputs $x,y \in \R^I$ and fix $\alpha \geq 12$ even integer. Since $\sum_{i=1}^{\infty}\frac{1}{2^i}\frac{|f^{(l)}(x)-f^{(l)}_i(y)|}{1+|f^{(l)}(x)-f^{(l)}_i(y)|}<\sum_{i=1}^{\infty}\frac{1}{2^i}=1$ and, by Jensen inequality, also $\sum_{i=1}^{\infty}\frac{1}{2^i} \big( \frac{|f^{(l)}(x)-f^{(l)}_i(y)|}{1+|f^{(l)}(x)-f^{(l)}_i(y)|} \big)^{\alpha}<\sum_{i=1}^{\infty}\frac{1}{2^i}=1$, we get

\[
\begin{split}
d\big(\textbf{F}^{(l)}(x),\textbf{F}^{(l)}(y)\big)_{\infty} ^{\alpha} &= \Big( \sum_{i=1}^{\infty}\frac{1}{2^i}\frac{|f_i^{(l)}(x)-f^{(l)}_i(y)|}{1+|f_i^{(l)}(x)-f^{(l)}_i(y)|}\Big)^{\alpha} \\
&\leq \sum_{i=1}^{\infty}\frac{1}{2^i} \Big( \frac{|f_i^{(l)}(x)-f^{(l)}_i(y)|}{1+|f_i^{(l)}(x)-f^{(l)}_i(y)|}\Big)^{\alpha} \\
&\leq \sum_{i=1}^{\infty}\frac{1}{2^i} |f_i^{(l)}(x)-f^{(l)}_i(y)|^{\alpha} \\
\end{split}
\]

Thus, by applying monotone convergence theorem to the positive increasing sequence $g(N)=\sum_{i=1}^{N}\frac{1}{2^i} |f_i^{(l)}(x)-f^{(l)}_i(y)|^{\alpha}$ (which allows to exchange $\E$ and $\sum_{i=1}^{\infty}$), we get

\[
\begin{split}
\E \Big[ d\big(\textbf{F}^{(l)}(x),\textbf{F}^{(l)}(y)\big)_{\infty} ^{\alpha} \Big] &\leq \E \Big[ \sum_{i=1}^{\infty}\frac{1}{2^i} |f^{(l)}(x)-f^{(l)}_i(y)|^{\alpha} \Big]\\
& = \E \Big[ \lim_{N \rightarrow \infty} \sum_{i=1}^{N}\frac{1}{2^i} |f_i^{(l)}(x)-f^{(l)}_i(y)|^{\alpha} \Big] \\
& = \lim_{N \rightarrow \infty} \E \Big[ \sum_{i=1}^{N}\frac{1}{2^i} |f_i^{(l)}(x)-f^{(l)}_i(y)|^{\alpha} \Big] \\
& = \sum_{i=1}^{\infty}\frac{1}{2^i} \E \Big[ |f_i^{(l)}(x)-f^{(l)}_i(y)|^{\alpha} \Big] \\
& = \sum_{i=1}^{\infty}\frac{1}{2^i} H^{(l)} \|x - y \|^{\alpha}_{\R^I} \\
& = H^{(l)} \|x - y \|_{\R^I}^{\alpha} 
\end{split}
\]
where we used (\ref{expec}) and the fact that $H^{(l)}$ does not depend on $i$ (see (\ref{H^l})). 

Therefore, by Proposition \ref{Kolmogorov_Chentsov_0}, for each $\alpha > I$, setting $\beta=\alpha-I$ (since $\beta$ needs to be positive, it is sufficient to choose $\alpha>I$) $\textbf{F}^{(l)}$ has a continuous version $\textbf{F}^{(l)(\theta)}$ and the latter is $\mathbb{P}$-a.s locally $\gamma$-H\"older continuous for every $0<\gamma<1-\frac{I}{\alpha}$.

\begin{itemize}
\item Thus $\textbf{F}^{(l)(\alpha)}$ and $\textbf{F}^{(l)}$ are indistinguishable (same trajectories), i.e there exists $\Omega^{(\alpha)}\subset \Omega$ with $\mathbb{P} (\Omega^{(\alpha)})=1$ such that for each $\omega \in \Omega^{(\alpha)}$, $x \mapsto \textbf{F}^{(l)}(x)(\omega)$ is locally $\gamma$-H\"older for each $0<\gamma<1-\frac{I}{\alpha}$.
\item Define $\Omega^{\star}= \bigcap_{\alpha > I}\Omega^{(\alpha)}$, then for each $0 < \delta_0 < 1$ there exists $\alpha_0$ such that $\delta_0 < 1- \frac{I}{\alpha_0} < 1$, thus for each $\omega \in \Omega^{\star} \subset \Omega^{(\alpha_0)}$, the trajectory $x \mapsto \textbf{F}^{(l)}(x)(\omega)$ is locally $\delta_0$-H\"older continuous.
\end{itemize}
By Proposition \ref{Kolmogorov_Chentsov_0} we can conclude that $\textbf{F}^{(l)}$ has a continuous version and the latter is $\mathbb{P}$-a.s locally $\gamma$-H\"older continuous for every $0<\gamma<1$.

\section*{SM E}\zlabel{sec:appE}
\subsection*{General introduction to Daniell-Kolmogorov extension theorem}
Let $X$ be a set of indexes and $\{(E_x, \mathcal{E}_x)\}_{x \in X}$ measurable spaces. On $E:=\times_{x \in X}E_x$ we can consider the $\sigma$-algebra $\mathcal{E}:= \bigotimes_{x \in X} \mathcal{E}_x$ that is 

\[
\mathcal{E} = \sigma (\pi_x, x \in X)=\sigma \Big( \bigcup_{x \in X} \pi_x^{-1}(\mathcal{E}_x) \Big)
\]

where for each $x \in X$, $\pi_x:E \rightarrow E_x, \omega:=(\omega_x)_{x\in X} \mapsto \pi_x(\omega)=\omega_x$. $\mathcal{E}$ is generated by measurable rectangles. A measurable rectangle $A$ is of the form
\[
A:=\times_{x \in X}A_x \text{ such that only a finite number of $A_x \in \mathcal{E}_x$ are different from $E_x$}
\]

\subsection*{$\sigma$-algebra on the space of functions}
Fix $X=\R^I$ and $(S,d)$ Polish space. We consider the measurable sets $\{(E_x, \mathcal{E}_x)\}_{x \in X} = \{(S, \mathcal{B}(S))\}_{x \in \R^I}$ thus we can construct a measurable space
\[
(E,\mathcal{E})= (\times_{x \in X}E_x,\bigotimes_{x \in X} \mathcal{E}_x)= \Big( S^{\R^I}, \mathcal{B}(S^{\R^I}) \Big)
\]

where $S^{\R^I}=\times_{x \in \R^{I}}S$ is the set of all functions from $\R^I$ into $S$ and \[
\begin{split}
\mathcal{B}(S^{\R^I})&:=\bigotimes_{x\in \R^I}\mathcal{B}(S) \\
&= \sigma \Big( \bigcup_{x \in \R^I} \pi_x^{-1}(\mathcal{B}(S)) \Big) \\
&= \sigma \Big( \Big\{A:=\times_{x \in X}A_x \text{ such that only a finite number of $A_x$ are different from $S$ \Big\}} \Big)
\end{split}
\]

An example of measurable rectangle is 
\[
A= S \times A_{x^{(1)}} \times S \times A_{x^{(2)}} \times S \times S \times \dots \times A_{x^{(k)}} \times S \times S \times \dots
\]
where $k \in \N$ and only for $x^{(1)}, \dots , x^{(k)}$ the cartesian products are different to $S$. 

Denote by $Z=(Z_x)_{x\in\R^I}$,  $Z_x:(\Omega,\mathcal{H},\P) \rightarrow S$ any stochastic process of interest, such as $f^{(l)}_i(n)$ or $f^{(l)}_i$ for some $l \geq 1$, $i\geq 1$ and $n \geq 1$ when $S=(\R,|\cdot|)$, or even $\mathbf{F}^{(l)}(n)$ or $\mathbf{F}^{(l)}$ for $l \geq 1$ and $n \geq 1$ when $S=(\R^{\infty}, \|\cdot \|_{\infty})$. Consider the finite-dimensional distributions of $Z$ 
\[
\Lambda=\{P^Z_{x^{(1)}, \dots, x^{(k)}} \text{ on } \mathcal{B}(S^{k})| x^{(j)} \in \R^I, j  \in \{1, \dots, k\}, k \in \N\}
\]
If $\Lambda$ is consistent in the sense of Kolmogorov theorem, then there exists an unique probability measure $\P'$ on $(S^{\R^I}, \mathcal{B}(S^{\R^I}))$ such that the canonical process $Z'=(Z'_x)_{x \in \R^I}$, $Z'_x:S^{\R^I} \rightarrow S, \omega \mapsto Z'_x(\omega)= \omega(x)$ on $(S^{\R^I}, \mathcal{B}(S^{\R^I}), \P')$ has finite-dimensional distributions that coincide with $\Lambda$.
\subsection*{SM E.1 : Existence of a probability measure on $S^{\R^I}$ for the sequence processes}
Fix $S=\R$. Fix a layer $l$, a unit $i\geq 1$ on that layer and $n \in \N$. We want to prove that there exists a probability measure $\P^{(i,l,n)}$ on $(\R^{\R^I}, \mathcal{B}(\R^{\R^I}))$ such that the associated canonical process $\Theta^{(i,l,n)}_x: \R^{\R^I} \to \R$, $\omega \mapsto \omega(x)$ has finite-dimensional distributions that coincide with
\[
\Lambda^{(i,l,n)} = \Big\{ P^{(i,l,n)}_{x^{(1)}, \dots , x^{(k)}} \Big\}_{k \in \N},
\]
where $P^{(i,l,n)}_{x^{(1)}, \dots , x^{(k)}}$ is the distribution of $f^{(l)}_{i}(\textbf{X},n)$. We do not know the exact form of this distribution but we know the distribution of the conditioned random variable $f^{(l)}_{i}(\textbf{X},n)|f^{(l-1)}_{1,\dots , n}$ (see (\ref{f1k})). Thus, since from (\ref{f1k}) the distribution of $f^{(1)}_{i}(\textbf{X})$ is well known, proceeding by induction it is sufficient to prove the existence of two probability measures $\P^{(i,1,n)}$ and $\P^{(i,l,n)}_{|l-1}$ on $(\R^{\R^I}, \mathcal{B}(\R^{\R^I}))$ such that the associated canonical processes $\Theta^{(i,1,n)}_x$, and $\Theta^{(i,1,n)|l-1}_x$ have finite-dimensional distributions that coincide respectively with
\[
\Lambda^{(i,1,n)} := \Big\{ P^{(i,1,n)}_{x^{(1)}, \dots , x^{(k)}} \Big\}_{k \in \N} \quad \text{ and } \quad \Lambda^{(i,l,n)}_{|l-1} := \Big\{ P^{(i,l,n,)|l-1}_{x^{(1)}, \dots , x^{(k)}} \Big\}_{k \in \N},
\]
where $P^{(i,1,n)}_{x^{(1)}, \dots , x^{(k)}}=N_k(\textbf{0},\Sigma(1,\textbf{X}))$ and $P^{(i,l,n)|l-1}_{x^{(1)}, \dots , x^{(k)}}=N_k(\textbf{0},\Sigma(l,n,\textbf{X}))$ defined on $\mathcal B(\R^{k})$. Observe that, for simplicity of notation, we have always avoided to write the dependence of the covariance matrix on the inputs matrix $\textbf{X}$, but in this case it is important to emphasize this. For the proof we defer to the limit case in the next subsection since the proof is the same step by step. When $S=\R^{\infty}$, recall that given a sequence of probability spaces $\{(\R^{\R^I}, \mathcal{B}(\R^{\R^I}), \P^{(i,l,n)})\}_{i\geq 1}$ there exists a unique probability measure $\P^{(l,n)}$ on $(\times_{i=1}^{\infty}\R^{\R^I}, \bigotimes_{i=1}^{\infty}\mathcal{B}(\R^{\R^I})) = \big((\R^{\infty})^{\R^I}, \mathcal{B}((\R^{\infty})^{\R^I})\big)$ such that, for each measurable rectangle $A= \times_{i=1}^{\infty}A_i$ where only for a finite number of $i$ the set $A_i$ is different from $\R^{\R^I}$, then $\P^{(l,n)}(A) = \prod_{i=1}^{\infty}\P^{(i,l,n)}(A_i)$. Moreover this probability is denoted as $\P^{(l,n)}=:\bigotimes_{i=1}^{\infty}\P^{(i,l,n)}$. This means that the existence of the stochastic processes $f^{(l)}_i(n)$ implies the existence of the stochastic processes $\mathbf{F}^{(l)}(n)$.

\subsection*{SM E.2 : Existence of a probability measure on $S^{\R^I}$ for the limit process}
Note that, as observed in previous section, the existence of the stochastic processes $f^{(l)}_i$ on $(\R^{\R^I}, \mathcal{B}(\R^{\R^I}))$ implies the existence of the stochastic processes $\textbf{F}^{(l)}$ on $\big((\R^{\infty})^{\R^I}, \mathcal{B}((\R^{\infty})^{\R^I})\big)$. Then we focus on the proof when $S=\R$. Fix a layer $l$ and a unit $i\geq 1$ on that layer. We want to prove that there exists a probability measure $\P^{(i,l)}$ on $(\R^{\R^I}, \mathcal{B}(\R^{\R^I}))$ such that the canonical process $\Theta^{(i,l)}_x:\R^{\R^I} \to \R$, $\omega \mapsto \omega(x)$ has finite-dimensional distributions that coincide with
\[
\Lambda^{(i,l)} = \Big\{ P^{(i,l)}_{x^{(1)}, \dots , x^{(k)}} \Big\}_{k \in \N},
\]
where $P^{(i,l)}_{x^{(1)}, \dots , x^{(k)}}$ are the finite-dimensional distributions of $f^{(l)}_i$ determined in (\ref{thm2}), i.e. $P^{(i,l)}_{x^{(1)}, \dots , x^{(k)}}=N_k(\textbf{0},\Sigma(l,\textbf{X}))$ defined on $\mathcal B(\R^{k})$. By Daniell-Kolmogorov existence result \cite[Theorem 6.16]{kallenberg2006foundations} it is sufficient to prove that for each $k \in \N$ and for each $x^{(1)}, \dots , x^{(k)}$ elements on $\R^I$, then 
\begin{align} \label{comp}
\begin{split}
&P^{(i,l)}_{x^{(1)}, \dots, x^{(z)}, \dots , x^{(k)}}(B^{(1)} \times \dots \times B^{(z-1)} \times \R \times B^{(z+1)} \times \dots \times B^{(k)} ) \\
&= P^{(i,l)}_{x^{(1)}, \dots, x^{(z-1)}, x^{(z+1)},\dots , x^{(k)}}(B^{(1)} \times \dots \times B^{(z-1)} \times B^{(z+1)} \times \dots \times B^{(k)} ),
\end{split}
\end{align}
for every $z \in \{1, \dots , k \}$ and for every $B^{(j)} \in \mathcal B(\R)$ for all $j=1, \dots , k$, $j \neq z$. Fix $k \in \N$, $k$ inputs $x^{(1)}, \dots , x^{(k)}$, $z \in \{1, \dots , k \}$ and $B^{(j)} \in \mathcal B(\R)$ for all $j=1, \dots , k$, $j \neq z$. Define the projection $\pi_{[z]}: \R^k \to \R^{k-1}$ such that $\pi_{[k]}(y_1, \dots , y_k)=[y_1, \dots ,y_{z-1}, y_{z+1}, \dots y_{k}]^T$. Thus, condition (\ref{comp}) is equivalent to the following:
\[
P^{(i,l)}_{x^{(1)}, \dots , x^{(k)}} \circ \pi_{[z]} = P^{(i,l)}_{\pi_{[z]}(x^{(1)}, \dots , x^{(k))}},
\]
where on the left we have the image measure of $P^{(i,l)}_{x^{(1)}, \dots , x^{(k)}}$ under $\pi_{[z]}$. We prove this by proving that the respective Fourier transformations coincide. In the following calculations we define $\textbf{y}=[y_1, \dots ,y_k]^T$, $\textbf{y}_{[z]}=[y_1, \dots, y_{z-1} , y_{z+1}, \dots , y_k]^T$ and $\textbf{t}=[t_1, \dots ,t_k]^T$, $\textbf{t}_{[k]}=[t_1,\dots , t_{z-1}, t_{z+1}, \dots ,t_k]^T$, then by definition of image measure we get
\[
\begin{split}
\varphi_{ \big(P^{(i,l)}_{x^{(1)}, \dots , x^{(k)}} \circ \pi_{[z]}\big) }(\textbf{t}_{[z]})&= \int_{\R^{k-1}} e^{ \ii  \textbf{t}_{[z]}^T \textbf{y}_{[z]}} \big(P^{(i,l)}_{x^{(1)}, \dots , x^{(k)}} \circ \pi_{[z]} \big) (\dd \textbf{y}_{[z]}) \\
&= \int_{\R^{k}} e^{ \ii  \textbf{t}_{[z]}^T \pi_{[z]}(\textbf{y})}P^{(i,l)}_{x^{(1)}, \dots , x^{(k)}} (\dd \textbf{y}).
\end{split}
\]
Now, recalling that $\textbf{1}_j$ is the $k \times 1$ vector with $1$ in the $j$-th position and $0$ otherwise, since $\pi_{[z]}(\textbf{y})=\textbf{y}_{[z]}$, defining $\pi^{\star}_{[z]}(\textbf{t})=\sum_{j=1 \\ j \neq z}^{k}\textbf{1}_jt_{j}$ we get $\textbf{t}_{[z]}^T \pi_{[z]}(\textbf{y})= \textbf{y}^T \pi^{\star}_{[z]}(\textbf{t})$. Then
\[
\begin{split}
\varphi_{ \big(P^{(i,l)}_{x^{(1)}, \dots , x^{(k)}} \circ \pi_{[z]}\big) }(\textbf{t}_{[z]})&= \int_{\R^{k}} e^{ \ii  \textbf{y}^T \pi^{\star}_{[z]}(\textbf{t})} P^{(i,l)}_{x^{(1)}, \dots , x^{(k)}} (\dd \textbf{y}) \\
&=\varphi_{P^{(i,l)}_{x^{(1)}, \dots , x^{(k)}}}(\pi^{\star}_{[z]}(\textbf{t}))\\
&=\varphi_{N_k(\textbf{0},\Sigma(l,\textbf{X}))}(\pi^{\star}_{[z]}(\textbf{t}))\\
&=\exp \big\{-\frac{1}{2} \pi^{\star}_{[z]}(\textbf{t})^T \Sigma(l,\textbf{X}) \pi^{\star}_{[z]}(\textbf{t}) \big\}\\
&=\exp \big\{-\frac{1}{2} \textbf{t}_{[z]}^T \widehat{\Sigma}(l,\textbf{X}) \textbf{t}_{[z]} \big\},
\end{split}
\]
where $\widehat{\Sigma}(l,\textbf{X})$ is the matrix $\Sigma(l,\textbf{X})$ without the $z$-th row and the $z$-th column. But since $\widehat{\Sigma}(l,\textbf{X})=\Sigma(l,\pi_{[z]}(x^{(1)}, \dots, x^{(k)}) )$ we get $\varphi_{ \big( P^{(i,l)}_{x^{(1)}, \dots , x^{(k)}} \circ \pi_{[z]} \big)}(\textbf{t}_{[z]}) = \varphi_{ P^{(i,l)}_{\pi_{[z]}(x^{(1)}, \dots , x^{(k)})} }(\textbf{t}_{[z]})$ for each $\textbf{t}_{[z]}$ and thus the two Fourier transformations coincides, as we wanted to prove.

\subsection*{SM E.3: Existence of a probability measure on $C(\R^I;\R)$}

If $Z$ is, in addition, a continuous stochastic process then we will show that there exists a probability measure $\P^Z$ on $C(\R^I;\R)\subset \R^{\R^I}$ endowed with a $\sigma$-algebra $\mathcal{G} \subset \mathcal{B}(\R^{\R^I})$ such that the finite-dimensional distribution of $Z'$ and $Z$ coincide. 

As suggested by \cite{kallenberg2006foundations} (page 311) we consider $C(\R^I;\R)$ with the topology of uniform convergence on compacts, that is
\begin{align}\label{metric_on_C}
\begin{sistema}
\rho_{\R}: C(\R^I;\R) \times C(\R^I;\R) \rightarrow [0, \infty),\\
\qquad (\omega_1,\omega_2) \mapsto \rho_S(\omega_1,\omega_2)=\sum_{R=1}^{\infty}\frac{1}{2^R}\sup_{x \in \overline{B}_R(0)} \xi(|\omega_1(x) -\omega_2(x)|_{\R})
\end{sistema}
\end{align}
The Borel $\sigma$-field  $\mathcal{G}:=\mathcal{B}(C(\R^I;\R),\rho_{\R})$ is generated by the evaluation maps $\pi_x$, thus it coincide with the product $\sigma$-field, i.e. $\mathcal{G}=\sigma(\Gamma)$, where
\[
\Gamma = \big\{ \Gamma_{x^{(1)}, \dots, x^{(k)}}(A) | A = A_{x^{(1)}} \times \dots \times A_{x^{(k)}}, A_{x^{(j)}} \in \mathcal{B}(\R), x^{(j)}\in \R^I, j \in \{1, \dots, k\},k \in \N \big\}
\] 
where $\Gamma_{x^{(1)}, \dots, x^{(k)}}(A)=\big\{ \omega \in C(\R^I;\R) | \omega(x^{(1)}) \in A_{x^{(1)}}, \dots , \omega(x^{(k)}) \in A_{x^{(k)}} \big\}$. 
Note that since $\sigma(\Gamma)\subset \mathcal{B}(\R^{\R^I})$ then $\mathcal{G}=\sigma(\Gamma) \subset \mathcal{B}(\R^{\R^I})$.

\begin{theorem}
There exists a unique probability measure $\P^Z$ on $(C(\R^I;\R),\mathcal{G})$ such that the canonical process $Z'$ restricted to $(C(\R^I;\R),\mathcal{G}))$ has finite-dimensional distributions that coincide with those of $Z$.
\end{theorem}

For the existence of $\P^Z$ consider the following 
\begin{lemma}
Let $(Z_x)_{x \in \R^I}$ be a $\R$-valued continuous stochastic process defined on $(\Omega, \mathcal{H}, \P)$. Then 
\[
\begin{sistema}
\mathcal{Z}: \Omega \rightarrow C(\R^I;\R) \\
\qquad \omega \rightarrow \mathcal{Z}(\omega)=(Z_x(\omega))_{x \in \R^I}
\end{sistema}
\]
is a random variable, i.e. measurable from $(\Omega,\mathcal{H})$ into $(C(\R^I;\R),\mathcal{G})$.
\end{lemma}
\begin{proof}
By previous proposition $\mathcal{G}= \sigma(\Gamma)$, then taking $\mathcal{O} \in \sigma(\Gamma)$, $\mathcal{O}=\Gamma_{x^{(1)}, \dots, x^{(k)}}(A)$ for some $k \in \N$, $\{x^{(1)}, \dots, x^{(k)}\} \subset \R^I$ and $A= A_{x^{(1)}} \times \dots \times A_{x^{(k)}}, A_{x^{(j)}} \in \mathcal{B}(\R)$,we get
\[
\begin{split}
\{\omega \in \Omega | \mathcal{Z}(\omega) \in \mathcal{O} \} &= \{\omega \in \Omega | Z_{x^{(1)}}(\omega) \in A_{x^{(1)}}, \dots ,Z_{x^{(k)}}(\omega) \in A_{x^{(k)}}\}\\
&= \bigcap_{j=1}^{k} \{ Z_{x^{(j)}} \in A_{x^{(j)}}\} \in \mathcal{H}
\end{split}
\]
where we used that $Z_{x^{(j)}}$ are random variables from $(\Omega, \mathcal{H})$ into $(\R,\mathcal{B}(\R))$.
\end{proof}

Then we can define a probability measure $\P^Z$ on $(C(\R^I;\R),\mathcal{G})$ being the image measure of $\mathcal{Z}$ under $\P$, that is 
\[
\forall \mathcal{O} \in \mathcal{G}, \quad \P^Z(\mathcal{O})=\P(\mathcal{Z} \in \mathcal{O})
\]

Now we prove that the finite-dimensional distributions of $Z'$ coincide with those of $Z$. It is sufficient to prove the following
\begin{lemma}
$\P^Z$ coincide wit the image measure of the canonical process $Z'$ under $\P^{'}$ restricted to $(C(\R^I;\R),\mathcal{G})$.
\end{lemma}

\begin{proof}
Fix $\mathcal{O} \in \mathcal{G}=\sigma(\Gamma)$, $\mathcal{O}=\Gamma_{x^{(1)}, \dots, x^{(k)}}(A)$ for some $k \in \N$, $\{x^{(1)}, \dots, x^{(k)}\} \subset \R^I$ and $A= A_{x^{(1)}} \times \dots \times A_{x^{(k)}}, A_{x^{(j)}} \in \mathcal{B}(\R)$. By definition of $\P^{Z}$,
\[
\begin{split}
    \P^Z(\mathcal{O})&=\P(\mathcal{Z} \in \Gamma_{x^{(1)}, \dots, x^{(k)}}(A))\\
    &=\P(\{\omega \in \Omega | \mathcal{Z}(\omega) \in \mathcal{O} \})\\
    &= \P(\{\omega \in \Omega | Z_{x^{(1)}}(\omega) \in A_{x^{(1)}}, \dots ,Z_{x^{(k)}}(\omega) \in A_{x^{(k)}})\\
    &=P^Z_{x^{(1)}, \dots , x^{(k)}}(A)
\end{split}
\]
By Daniell-Kolmogorv extension theorem the finite-dimensional distributions of $Z$ coincide with those of the canonical process $Z'$ under $\P'$, then $P^Z_{x^{(1)}, \dots , x^{(k)}}(A)= \P'(Z' \in \mathcal{O})$.
\end{proof}
The uniqueness of $\P^Z$ follows by the uniqueness $\P'$.

\subsection*{SM E.4: $\sigma (\times_{i=1}^{\infty}C(\R^I;\R)) \subset \sigma (C(\R^I;\R^{\infty})$}
First, note that $\times_{i=1}^{\infty}C(\R^I;\R) \simeq C(\R^I;\R^{\infty})$, indeed the map
\[
\Xi : C(\R^I;\R^{\infty}) \rightarrow \times_{i=1}^{\infty}C(\R^I;\R), \qquad \omega \mapsto (\omega_1, \omega_2, \dots)  
\]
is an isomorphism because is linear and bijective, indeed $\omega$ is $\|\cdot\|_{\infty}$-continuous if and only if each component $\omega_i$ is $|\cdot|$-continuous. It means that each element in one space could be seen as an element in the other and vice-versa, but different topologies are defined on these spaces. Now we prove that the sigma algebra generated by the product topology in $\times_{i=1}^{\infty}C(\R^I;\R)$ is contained on the sigma algebra generated by the topology of uniform convergence on compact set in $C(\R^I;\R^{\infty})$. For each $f,g \in C(\R^I;\R^{\infty})$ we have the following distances
\begin{align}\label{dist_infinity}
\begin{sistema}
\rho_{prod}(f,g)= \sum_{i=1}^{\infty}\frac{1}{2^i} \xi\Big(\sum_{R=1}^{\infty}\frac{1}{2^R} \sup_{x\in B_R(0)}\xi(|f_i(x)-g_i(x)|) \Big), \quad \text{ on } \times_{i=1}^{\infty}C(\R^I;\R)\\
\rho_{unif}(f,g)= \sum_{R=1}^{\infty}\frac{1}{2^R} \sup_{x\in B_R(0)}\xi\Big(\sum_{i=1}^{\infty}\frac{1}{2^i} \xi(|f_i(x)-g_i(x)|) \Big), \quad \text{ on } C(\R^I,\R^{\infty})
\end{sistema}
\end{align}

Using that $\xi$ is increasing and continuous and that $\sup_x (\sum_i h_i(x)) \leq \sum_i \sup_x h_i(x)$ it can be proved that there exists a constant $C>0$ such that $\| f \|_{unif}  \leq C \| f \|_{prod}$. This mean that if $h \in B^{prod}_{\epsilon}(f)=\{g: \|f-g\|_{prod}<\epsilon\}$ than $h \in B^{unif}_{C \epsilon}(f)=\{g :\|f-g\|_{unif}<\epsilon \}$, that is $B^{prod}_{\epsilon}(f) \subset B^{unif}_{C \epsilon}(f)$ which implies $\sigma(\rho_{prod})\subset \sigma(\rho_{unif})$. In particular each compact with respect to $\| \cdot \|_{prod}$ is compact with respect to $\| \cdot \|_{unif}$, indeed considering a $\| \cdot \|_{prod}$-compact $K$ then for every sequence $(k_i)\subset K$ there exists $(k_{i_j})\subset K$ and $k \in K$ such that $\|k_{i_j}-k \|_{prod} \rightarrow 0$. Moreover $\|k_{i_j}-k \|_{unif} \leq C \|k_{i_j}-k \|_{prod} \rightarrow 0$, i.e. $K$ is compact with respect to $\| \cdot \|_{unif}$.

\section*{SM F}\zlabel{sec:appF}
In this section we prove the Proposition \ref{Step1}.
\begin{proof}
By Proposition 16.6 of \cite{kallenberg2006foundations} $f(n) \overset{d}{\rightarrow} f$ in $C(\R^I;S)$ iff $f(n) \overset{d}{\rightarrow} f$ in $C(K;S)$ for any $K\subset \R^I$ compact. By Lemma 16.2 of of \cite{kallenberg2006foundations} the latter holds iff $f(n) \overset{f_d}{\rightarrow} f$ and $(f(n))_{n \geq 1}$ is relatively compact in distribution in $C(K;S)$. Note that converge of the finite-dimensional distributions holds in $\R^I$ iif it holds in the restriction $K$ for any compact $K \subset \R^I$. The space $(C(K;S), \rho_K)$, namely the space of continuous functions from a generic compact $K \subset \R^I$ to a Polish space $S$ and $C(K;S)$ endowed with the  uniform metric $\rho_K(f,g) = \sup_{x \in K} d(f(x),g(x))$, is itself a Polish space \citep[Lemma 3.97 and Lemma 3.99]{charalambos2013infinite}. Thus by Proposition 16.3 of \cite{kallenberg2006foundations}, i.e. Prohorov Theorem, on $C(K,S)$ $(f(n))_{n \geq 1}$ is relatively compact in distribution iif $(f(n))_{n \geq 1}$ is uniformly tight. Thus, so far we have shown that $f(n)\stackrel{d}{\rightarrow}f$ in $C(\R^I; S)$ iff: i) $f(n)\stackrel{f_d}{\rightarrow}f$ and ii) the sequence $(f(n))_{n \geq 1}$ is uniformly tight on $C(K,S)$ for a generic compact $K \subset \R^I$. It remains to show that the latter holds if $(f(n))_{n \geq 1}$ is uniformly tight on $C(\R^I;S)$. Fix $K$ compact in $\R^I$ and 
Consider the map 
\[
\pi_K:(C(\R^I;S),\rho_S) \rightarrow (C(K;S),\rho_K), \quad f \rightarrow f_{|K}
\]
where $f_{|K}$ is the restriction of $f$ to $K$ and $\rho_S$ is the metric $\rho_{\R}$ defined in (\ref{metric_on_C}) when $S=\R$ and $\rho_{unif}$ defined in (\ref{dist_infinity}) when $S=\R^{\infty}$. By proposition 16.4 of \cite{kallenberg2006foundations} if $\pi_K$ is continuous then it moves uniformly tight sequences into uniformly tight sequences. The continuity of $\pi_K$ follows by the proof of Proposition 16.6 of  \cite{kallenberg2006foundations}.
\end{proof}

\end{document}